\documentclass[a4paper,reqno,11pt]{amsart}
\usepackage{amsfonts,amsmath,amsthm,amssymb,stmaryrd}
\usepackage{mathrsfs}
\newtheorem{theorem}{Theorem}[section]
\newtheorem{lemma}[theorem]{Lemma}
\newtheorem{dfn}[theorem]{Definition}
\newtheorem{Proposition}[theorem]{Proposition}
\newtheorem{cor}[theorem]{Corollary}
\newtheorem{Remark}[theorem]{Remark}
\numberwithin{equation}{section}
\usepackage{color}
\usepackage[left=1 in, right=1 in,top=1 in, bottom=1 in]{geometry}

\providecommand{\abs}[1]{\left\vert#1\right\vert}
\providecommand{\norm}[1]{\left\Vert#1\right\Vert}

\providecommand{\Rn}[1]{\mathbb{R}^{#1}}

\providecommand{\br}[1]{\left\langle #1 \right\rangle}

\providecommand{\ns}[1]{\norm{#1}^2}

\def\nab{\nabla}
\def\dt{\partial_t}
\def\hal{\frac{1}{2}}

\def\ls{\lesssim}
\def\p{\partial}

\def\sg{\mathbb{D}}

\def\da{\Delta_{\mathcal{A}}}
\def\naba{\nab_{\mathcal{A}}}
\def\diva{\diverge_{\mathcal{A}}}

\def\H1{{_0}H^1(\Omega)}

\def\a{\mathcal{A}}

\def\n{\mathcal{N}}

\def\R{\mathbb{R}}

\DeclareMathOperator{\diverge}{div}

\providecommand{\Rn}[1]{\mathbb{R}^{#1}}
\providecommand{\norm}[1]{\left\Vert#1\right\Vert}

\title[Decay of passive scalars]{Passive scalars, moving boundaries, and Newton's law of cooling}

\author{Juhi Jang}
\address{
Department of Mathematics\\
University of California, Riverside\\
Riverside, CA 92521, USA
}
\email[J. Jang]{juhijang@math.ucr.edu}
\thanks{J. Jang was supported in part by NSF grant DMS-1212142}

\author{Ian Tice}
\address{
Department of Mathematical Sciences\\
Carnegie Mellon University\\
Pittsburgh, PA 15213, USA
}
\email[I. Tice]{iantice@andrew.cmu.edu}


\begin{document}

\begin{abstract}
We study the evolution of passive scalars in both rigid and moving slab-like domains, in both horizontally periodic and infinite contexts.  The scalar is required to satisfy Robin-type boundary conditions corresponding to Newton's law of cooling, which lead to nontrivial equilibrium configurations.  We study the equilibration rate of the passive scalar in terms of the parameters in the boundary condition and the equilibration rates of the background velocity field and moving domain.

\end{abstract}

\maketitle

%======================
%======================
\section{Introduction and motivation}
%======================
%======================

In this paper we study the passive scalar equation
\begin{equation}\label{passive_s}
\partial_t \theta + u \!\cdot\! \nabla \theta = \kappa \Delta \theta,
\end{equation}
where $\theta=\theta(t,x) \in \mathbb{R}$ measures some  scalar quantity at time $t\ge 0$ and position $x\in \Omega(t)$, $u(t,x)\in \mathbb{R}^3$ is a given divergence-free velocity field, and $\kappa > 0$ is the diffusivity.  Here $\Omega(t)$ is a three-dimensional open domain that may depend on time.  A full specification of $\Omega(t)$ will be provided later.   The scalar $\theta$ is   understood to be passive in the sense that $u$ affects the dynamics of $\theta$ through \eqref{passive_s}, but $\theta$ plays no role in the dynamics of $u$: it is passive.  We shall think of $u$ as the velocity field of an incompressible fluid evolving in the (possibly moving) domain $\Omega(t)$ and $\theta$ as the temperature of the fluid.

Passive scalars are a popular model of turbulent diffusion (see the exhaustive survey of Majda-Kramer \cite{majda_kramer} and references therein),  mixing phenomena (see for instance Lin-Thiffeault-Doering \cite{ltd} or Thiffeault \cite{thiffeault}), as well as pollution and combustion (see Warhaft \cite{warhaft} and references therein).  If we consider \eqref{passive_s} on $\mathbb{T}^3$, then, as we will show below,  the solution $\theta$ will converge exponentially as $t \to \infty$ to its average, which is an equilibrium solution of \eqref{passive_s}.  An interesting  result of Constantin-Kiselev-Ryzhik-Zlato\v{s} \cite{ckrz} shows that certain flows $u$ can actually enhance this convergence to equilibrium.  This was extended to time-periodic flows by Kiselev-Shternberg-Zlato\v{s}  \cite{ksz} and to flows on  $\Rn{2}$ by Zlato\v{s} \cite{zlatos} (in which case the equilibrium is $0$).

We aim to study the temporal decay properties of the scalar $\theta$ in terms of the decay properties of the given velocity $u$, the asymptotic behavior of $\Omega(t)$, and the the boundary conditions imposed on $\theta$ and $u$.  We investigate the decay rate of $\theta$ for long times and also show that the decay leads to global-in-time solutions for arbitrary data.

%======================
\subsection{Periodic boundary conditions}
%======================

To motivate our analysis let's first consider a relatively trivial example,  when $\Omega$ is taken to be a periodic box $\mathbb{T}^3 = \mathbb{R}^3 / \mathbb{Z}^3$ and $u(t,x)$ is a divergence-free, periodic, smooth vector field.  In this case, by using the periodicity and the incompressibility condition (i.e.  $\diverge u=0$), it is easy to see that a solution $\theta$ to \eqref{passive_s} obeys the $L^2$-energy law
\begin{equation}\label{energy_0}  
\frac12\frac{d}{dt}\int_{\mathbb{T}^3} \abs{\theta}^2 dx + \kappa \int_{\mathbb{T}^3} |\nabla\theta|^2 dx=0. 
\end{equation}
Similarly, we may integrate \eqref{passive_s} to find that
\begin{equation}
\frac{d}{dt}\int_{\mathbb{T}^3} \theta dx=0. 
\end{equation}
Thus  
\begin{equation}
\theta_{ave} =\frac{1}{| \mathbb{T}^3 |} \int_{\mathbb{T}^3} \theta dx = \int_{\mathbb{T}^3} \theta dx
\end{equation}
is a constant determined by the initial average of the temperature.  Since $\theta - \theta_{ave}$ also solves \eqref{passive_s} with periodic boundary conditions, we also know that
\begin{equation}  \label{energy_1} 
\frac12\frac{d}{dt}\int_{\mathbb{T}^3} \abs{ \theta-\theta_{ave}}^2 dx + \kappa \int_{\mathbb{T}^3} |\nabla\theta|^2 dx=0. 
\end{equation}
We now recall the (sharp) Poincar\'{e} inequality in $\mathbb{T}^3$: 
\begin{equation}\label{torus_poincare}
\int_{\mathbb{T}^3} \abs{ \theta-\theta_{ave}}^2 dx \leq \frac{1}{4\pi^2} \int_{\mathbb{T}^3} |\nabla\theta|^2 dx. 
\end{equation}
Hence, by combining  \eqref{energy_1} and \eqref{torus_poincare}, we have 
\begin{equation}
\frac12\frac{d}{dt}\int_{\mathbb{T}^3} \abs{ \theta-\theta_{ave}}^2 dx + 4\pi^2 \kappa  \int_{\mathbb{T}^3} \abs{ \theta-\theta_{ave}}^2 dx \leq 0,  
\end{equation}
which in turn implies that
\begin{equation}\label{periodic_decay}
\| \theta(t,\cdot)-\theta_{ave} \|_{L^2}\leq e^{-4\pi^2 \kappa t }  \| \theta(0,\cdot)-\theta_{ave} \|_{L^2}. 
\end{equation}
Therefore, we deduce that $\theta$ converges exponentially fast to $\theta_{ave}$.  Since the inequality \eqref{torus_poincare} is sharp, this decay result is optimal.

We would like to highlight three important features of this decay analysis.  First, the periodic boundary conditions on $\theta$ play a role in the decay through the derivation of \eqref{energy_1}.  Second, the divergence-free condition and the periodic boundary conditions on $u$ also play a role in the derivation of \eqref{energy_1}; in this particular setting $u$ actually disappears entirely from \eqref{energy_1} and has no ultimate impact on the decay rate in \eqref{periodic_decay}.  Third, the fact that $\Omega(t) = \mathbb{T}^3$ is static in time gives rise to a static constant in the Poincar\'{e} inequality \eqref{torus_poincare}, which in turn plays a serious role in determining the decay rate in \eqref{periodic_decay}.  Indeed, the Poincar\'e constant and the diffusion constant $\kappa$ both appear directly in the decay rate.

Throughout the rest of this paper we will investigate how changes in these three features affect the decay of $\theta$.  We will consider both static (and rigid) domains as well as moving domains with a range of boundary conditions for $\theta$ and $u$.    In the rest of the section we will give concrete examples of couplings between $u$ and $\theta$ in order to provide motivation for our main results.  However, in the main results of Section \ref{sec_main}, we will consider a more general and weaker coupling between $u$ and $\theta$.

%======================
\subsection{Rigid boundary}
%======================

We consider a fluid confined between two rigid horizontal plates located at $x_3=0$ and $x_3=-d$ for $d >0$.   We allow the horizontal cross-section to be either infinite, in which case we define $\Gamma = \mathbb{R}^2$, or else periodic, in which case we set $\Gamma = (L_1 \mathbb{T}) \times (L_2 \mathbb{T})$, where $L_1,L_2 >0$ are the periodicity lengths and $L \mathbb{T} = \R / (L \mathbb{Z})$.  Then the rigid domain is
\begin{equation}\label{omega_def}
 \Omega = \Gamma \times (-d,0).
\end{equation}
We shall write 
\begin{equation}\label{sigma_def}
 \Sigma_+ = \{x_3 = 0\} \text{ and } \Sigma_- = \{x_3 = -d\}
\end{equation}
to denote the top and bottom boundaries of $\Omega$.

We assume that the fluid obeys the gravity-driven incompressible Navier-Stokes equations
\begin{equation}\label{rigid_ns}
\begin{cases}
 \partial_t u + u \!\cdot \!\nabla u + \nabla p =\mu \Delta u - g e_3 & \text{in } \Omega \\
 \diverge u = 0 & \text{in } \Omega \\
 u=0 & \text{on } \Sigma_+ \text{ and } \Sigma_-.
 \end{cases}
\end{equation}
Here $\mu>0$ is the fluid viscosity, $g >0$ is the gravitational strength, and $e_3 = (0,0,1) \in \R^3$.   Actually, the gravitational force plays no real role in \eqref{rigid_ns} since we may absorb this conservative force into the pressure by swapping $p \mapsto p + g x_3$.  We have included gravity here simply to maintain a similarity to the moving boundary case, where gravity does play a role.  

We assume that the temperature of the fluid satisfies the passive scalar equation  \eqref{passive_s} in $\Omega$.   Notice that $\theta$ does not appear in \eqref{rigid_ns}: it is passive in the dynamics of $u$.  However, the velocity $u$ does play a role in \eqref{passive_s}.  It remains to specify boundary conditions for $\theta.$    

Since we are thinking of $\theta$ as the temperature of the fluid, there is a natural continuum of choices for boundary conditions at each plate.  We will assume that the space adjacent to each plate is maintained at a constant temperature, with $\theta_{ext,-}$ and $\theta_{ext,+}$ denoting the temperature below the bottom plate and the temperature above the top plate, respectively.  One possibility is to  suppose that the plate conducts heat perfectly into the fluid, which requires the temperature to match the temperature external to the plate.  This is modeled through a Dirichlet condition $\theta = \theta_{ext}$.  A more general condition supposes that the plate is an imperfect heat conductor.  In this case it is common (see for example Slattery's book \cite{slattery} for a discussion of these conditions for heat conduction and mass transfer) to model the plate with Newton's law of cooling, which requires the Robin-type boundary condition 
\begin{equation}
\kappa \nab \theta \cdot \nu = \beta(\theta_{ext} - \theta),
\end{equation}
where $\beta \in [0,\infty]$ is the heat transfer coefficient and $\nu$ is the outward-pointing unit normal on the plate.  Notice when $\beta = \infty$ we recover the Dirichlet condition $\theta = \theta_{ext}$; when $\beta=0$ we recover what is known as an insulating condition, corresponding to a homogeneous Neumann boundary condition.     

We may reduce the number of parameters to be considered by studying the evolution of $\theta - \theta_{ext,-}$ instead of $\theta$.  The equation \eqref{passive_s} is still satisfied, but this allows us to replace the temperature below the lower plate by $0$ and the temperature above the upper plate by 
\begin{equation}
\bar{\theta} := \theta_{ext,+} - \theta_{ext,-} \in \mathbb{R}. 
\end{equation}
Since the materials above and below the fluid may in principle be different, we will allow for different heat transfer coefficients in Newton's law of cooling at the top and bottom.  We summarize the equations for $\theta$ as follows: 
\begin{equation}\label{passive_rigid2}
\begin{cases}
\partial_t \theta + u \!\cdot\! \nabla \theta = \kappa \Delta \theta & \text{in } \Omega  \\
\kappa \partial_{x_3}\theta=\beta_+(\bar{\theta}-\theta) & \text{on } \Sigma_+  \\
\kappa \partial_{x_3}\theta=\beta_- \theta & \text{on } \Sigma_-.
\end{cases}
\end{equation}
Here we have used the fact that $\nu = \pm e_3$ on $\Sigma_\pm$.  Henceforth we will write $\beta = (\beta_+,\beta_-) \in [0,\infty]^2$ for the pair of heat conductivity coefficients.

%======================
\subsection{Moving boundary}
%======================

We shall also  consider a fluid evolving in a moving domain $\Omega(t)$.  Since we wish to compare the decay properties of $\theta$ in moving domains and in rigid domains, we will suppose that $\Omega(t)$ is of the same general ``slab-like'' form as the domain $\Omega$ considered above.   We again assume $\Gamma = \mathbb{R}^2$ or $\Gamma = (L_1\mathbb{T})\times (L_2 \mathbb{T})$ to allow for infinite or periodic horizontal cross sections.  We then assume that
\begin{equation}
\Omega(t) = \{ y \in \Gamma \times \mathbb{R}\ | \ -d<y_3<\eta(t,y_1, y_2)\}, 
\end{equation}
where $\eta: \mathbb{R}_+ \times \Gamma \to \R$ is the ``free surface function.''  Then 
\begin{equation}
 \Sigma(t) = \{ y_3 = \eta(t,y_1,y_2)\}
\end{equation}
is the free surface of the fluid, while 
\begin{equation}
\Sigma_- =  \{y_3 = -d\}
\end{equation}
is the rigid bottom of the fluid domain.

The motion of the fluid is governed by the gravity-driven incompressible Navier-Stokes system in $\Omega(t)$: 
\begin{equation}\label{moving_u}
\begin{cases}
 \partial_t u + u \!\cdot \!\nabla u + \nabla p =\mu \Delta u - ge_3  & \text{in } \Omega(t) \\
 \diverge u = 0 & \text{in } \Omega(t) \\
 \partial_t \eta = u_3 - u_1 \partial_{y_1}\eta- u_2\partial_{y_2}\eta & \text{on } \Sigma(t)\\
\left( p I - \mu \mathbb{D}(u)\right)\nu =  - \sigma {H}\nu & \text{on } \Sigma(t) \\
 u=0 & \text{on } \Sigma_-.
 \end{cases}
\end{equation}
Here $\mu>0$ is the fluid viscosity, $\nu$ is the outward-pointing unit normal on $\Sigma(t)$, $I$ is the $3\times 3$ identity matrix, $\left( \mathbb{D}u \right)_{ij}= \partial_i u_j +\partial_j u_i$ is the symmetric gradient of $u$, $g>0$ is the gravitational constant, $e_3 = (0,0,1)$,  $\sigma\geq 0$ is the surface tension coefficient, and $H$ is twice the mean curvature of the surface $\Sigma(t)$.  

The dynamics of the temperature of the fluid are described by the convection-diffusion equation \eqref{passive_s} with boundary conditions of the form considered above: 
\begin{equation}\label{moving}
\begin{cases}
 \partial_t \theta + u \!\cdot\! \nabla \theta = \kappa \Delta \theta & \text{in } \Omega(t), \\
 \kappa \nab \theta \cdot \nu  =\beta_+ (\bar{\theta} -\theta )  &\text{on } \Sigma(t)  \\
\kappa \partial_{x_3}\theta=\beta_- \theta & \text{on } \Sigma_-.
 \end{cases}
\end{equation}
As in the rigid boundary case, we will consider $\beta =(\beta_+,\beta_-) \in [0,\infty]^2$.  Notice that $\theta$ is passive in the dynamics of $u$: it does not appear in \eqref{moving_u}.   On the other hand, both the velocity field $u$ and the changing boundary (i.e. $\eta$) affect the dynamics of $\theta$.

\subsubsection{Reformulation in flattened coordinates}

In  moving boundary  problems, it is convenient to flatten the free surface by using a suitable coordinate transformation.  We will utilize a flattening coordinate transformation introduced by Beale \cite{B2}. To this end, we consider the fixed domain $\Omega$ given by  \eqref{omega_def}, for which we will write the coordinates as $x\in \Omega$.  We shall again write $\Sigma_+= \{x_3=0\}$ for the upper boundary and $\Sigma_- = \{x_3=-d\}$ for the lower boundary.  We view $\eta$ as a function on $\mathbb{R}_+\times \Gamma$.  We then define 
\begin{equation}
 \bar{\eta}:= \mathcal{P} \eta 
 = \text{harmonic extension of }\eta \text{ into the lower half space}
\end{equation}
where $\mathcal{P} $ is defined in the appendix: by \eqref{poisson_def_inf} when $\Gamma = \R^2$ and by \eqref{poisson_def_per} when $\Gamma = (L_1 \mathbb{T}) \times (L_2 \mathbb{T})$.   The harmonic extension $\bar{\eta}$ allows us to flatten the coordinate domain via the mapping
\begin{equation}\label{mapping_def}
 \Omega \ni x \mapsto   (x_1,x_2, x_3 +  \bar{\eta}(t,x)(1+ x_3/d)) = \Phi(t,x) = (y_1,y_2,y_3) \in \Omega(t).
\end{equation}
Note that $\Phi(t, \Sigma_+) = \{ y_3 = \eta(t,y_1,y_2) \} = \Sigma(t)$ and $\Phi(t,\cdot)\vert_{\Sigma_-} = Id_{\Sigma_-}$, i.e. $\Phi$ maps $\Sigma_+$ to the free surface $\Sigma(t)$ and keeps the lower surface fixed.   We write
\begin{equation}\label{A_def}
 \nab \Phi = 
\begin{pmatrix}
 1 & 0 & 0 \\
 0 & 1 & 0 \\
 A & B & J
\end{pmatrix}
\text{ and }
 \mathcal{A} := (\nab \Phi^{-1})^T = 
\begin{pmatrix}
 1 & 0 & -A K \\
 0 & 1 & -B K \\
 0 & 0 & K
\end{pmatrix}
\end{equation}
for 
\begin{equation}\label{ABJ_def}
\begin{split}
A &= \p_1 \bar{\eta} \tilde{d} ,\;\;\;  B = \p_2 \bar{\eta} \tilde{d} ,  \;\;\; 
J =  1+ \bar{\eta}/d + \p_3 \bar{\eta} \tilde{d},  \;\;\; K = J^{-1},\;\;\;  \tilde{d}  = (1+x_3/d).  
\end{split}
\end{equation} 
Here $J = \det{\nab \Phi}$ is the Jacobian of the coordinate transformation.

If $\eta$ is sufficiently small (in an appropriate Sobolev space), then the mapping $\Phi$ is a $C^1$ diffeomorphism.  This allows us to transform the problem to one on the fixed spatial domain $\Omega$ for $t \ge 0$.  In the new coordinates, the PDE \eqref{moving_u} becomes
\begin{equation}\label{geometric}
 \begin{cases}
  \dt u - \dt \bar{\eta} \tilde{d} K \p_3 u + u \cdot \naba u -  \mu \da u + \naba p     =0 & \text{in } \Omega \\
 \diva u = 0 & \text{in }\Omega \\
 \dt \eta = u \cdot \n & \text{on } \Sigma \\
 (p I  - \mu \sg_{\a} u) \n = g\eta \n -\sigma H \n & \text{on } \Sigma \\
 u = 0 & \text{on } \Sigma_b
 \end{cases}
\end{equation}
with  given initial data $u(0,x) = u_0(x)$, $ \eta(0,x') = \eta_0(x')$  where $x'=(x_1,x_2)$.  We will employ the Einstein convention of summing over  repeated indices for vector and tensor operations.   This allows us to define the differential operators used in \eqref{geometric} as follows.   The actions of $\naba$, $\diva$, and $\da$ are given by $(\naba f)_i := \a_{ij} \p_j f$, $\diva X := \a_{ij}\p_j X_i$, and $\da f = \diva \naba f$ for appropriate $f$ and $X$.   For $u\cdot \naba u$ we mean $(u \cdot \naba u)_i := u_j \a_{jk} \p_k u_i$.  We write   $(\sg_{\a} u)_{ij} = \a_{ik} \p_k u_j + \a_{jk} \p_k u_i$ for the symmetric $\a-$gradient.
We have also written  
\begin{equation}\label{N_def}
\n := -\p_1 \eta e_1 - \p_2 \eta e_2 + e_3 
\end{equation}
for the non-unit normal to $\Sigma$.

Similarly, in the new coordinates, the PDE \eqref{moving} reads 
\begin{equation}\label{geometric1}
 \begin{cases}
  \dt \theta - \dt \bar{\eta} \tilde{d} K \p_3 \theta + u \cdot \naba \theta -\kappa \da \theta      =0 & \text{in } \Omega \\
 \kappa \naba \theta \cdot \nu= \beta_+(\bar{\theta}-\theta)   & \text{on } \Sigma_+ \\
  \kappa \partial_{x_3} \theta = \beta_-\theta   & \text{on } \Sigma_- 
 \end{cases}
\end{equation}
with given initial data $\theta(0,x)=\theta_0(x)$. Here $\nu=\frac{\n} {|\n|}$. Notice that all the differential operators in \eqref{geometric} and \eqref{geometric1} are connected to $\eta$ and thus to  the geometry of the free surface. 

%======================
\subsection{Equilibria}
%======================

An equilibrium solution corresponds to $u =0$ and $\Omega(t) = \Omega$ given by \eqref{omega_def} (and hence $\eta =0$ in the case of a moving boundary).  In this case  \eqref{passive_rigid2} and  \eqref{geometric1} reduce to the same equations, so the equilibria coincide for the rigid and moving boundaries.  To denote the dependence on $\beta = (\beta_+,\beta_-)$ let us denote the equilibrium temperature by $\theta_{eq}^\beta$.  We need to solve $\Delta \theta_{eq}^\beta=0$ with the given boundary conditions.   

The general equilibrium solution is then 
\begin{equation}\label{eq}
 \begin{split}
\theta_{eq}^\beta &= \frac{\kappa \beta_+ \bar{\theta} + \beta_+\beta_- \bar{\theta} (x_3 + d)}{\kappa(\beta_+ + \beta_-) + \beta_+ \beta_- d}, \text{ for } \beta \in [0,\infty]^2 \backslash \{(0,0)\}, \\
\theta_{eq}^\beta &= C = \text{some constant, for } \beta = (0,0).  
 \end{split}
\end{equation}
 Notice in particular that the equilibrium with either $\beta_+ = \infty$ or $\beta_- = \infty$, but not both, is recovered from \eqref{eq} by taking a limit:
\begin{equation}
\theta_{eq}^{(\beta_+,\infty)}= \frac{\beta_+ \bar{\theta}}{\kappa+\beta_+ d} \left( x_3 + d\right) \text{ and } \theta_{eq}^{(\infty, \beta_-)}= \bar{\theta} + \frac{\beta_- \bar{\theta}}{\kappa + \beta_- d}x_3. 
\end{equation}
The equilibrium with $\beta_\pm =\infty$ (pure Dirichlet boundary conditions) is recovered through a further limit:
\begin{equation}
 \theta_{eq}^{(\infty,\infty)} = \frac{\bar{\theta}(x_3 + d)}{d}.
\end{equation}

We note that unlike in the case of a periodic box, the equilibrium temperature is in general not a constant because of the boundary conditions.  We will see that solutions of \eqref{passive_rigid2} and \eqref{geometric1} converge asymptotically to these equilibrium solutions when $\beta \neq (0,0)$.  When $\beta =(0,0)$ the equilibrium is not unique; this is related to the fact that solutions converge (at least in the periodic case) to a constant determined by the initial data.

%======================
\section{Main results}\label{sec_main}
%======================

%======================
\subsection{Well-posedness}
%======================

We begin with a discussion of the notion of weak solutions to \eqref{passive_rigid2} and \eqref{geometric1}, given as perturbations of the equilibrium state $\theta_{eq}^\beta$.  Since $\theta_{eq}^\beta$ is not in $L^2(\Omega)$ when $\Omega$ is horizontally infinite, we must formulate the weak problem in terms of the perturbation.  We will also discuss the construction of local-in-time solutions.

First we define the appropriate functional setting for solutions.  Let $H^1_\beta(\Omega)$ denote the space 
\begin{equation}\label{sobolev_beta}
H^1_\beta(\Omega)=
 \begin{cases}
  H^1(\Omega) &\text{if } \beta \in [0,\infty)^2 \\
  \{\varphi \in H^1(\Omega) \;\vert\; \varphi\vert_{\Sigma_+} =0  \}  &\text{if } \beta_+ = \infty, \beta_- \in [0,\infty)  \\
  \{\varphi \in H^1(\Omega) \;\vert\; \varphi\vert_{\Sigma_-} =0  \}  &\text{if } \beta_- = \infty, \beta_+ \in [0,\infty)  \\
  \{\varphi \in H^1(\Omega) \;\vert\; \varphi\vert_{\Sigma_+} = \varphi\vert_{\Sigma_-} =0  \}  &\text{if } \beta_+ = \beta_-= \infty.
 \end{cases}
\end{equation}

Now we give the notion of weak solutions for the rigid problem.  

\begin{dfn}\label{weak_rigid_def}
Let $T \in (0,\infty]$.  Suppose that $u: [0,T] \times \Omega \to \R^3$  is a given field satisfying 
\begin{equation}\label{rigid_assump_1}
 u \in L^\infty([0,T];L^\infty(\Omega)) \cap L^2([0,T];L^2(\Omega)).
\end{equation}
Let $\theta_0 \in L^2(\Omega)$.

We say that a a function $\theta: [0,T] \times \Omega \to \R$ is a weak solution to \eqref{passive_rigid2} if the following three conditions are satisfied.
\begin{itemize}
 \item[(i)] We have the inclusions 
\begin{equation}
 \theta - \theta_{eq}^\beta \in L^2([0,T]; H^1_\beta(\Omega)) \text{ and } \dt \theta \in L^2([0,T]; (H^1_\beta(\Omega))^\ast ).
\end{equation}

 \item[(ii)] $(\theta - \theta_{eq}^\beta) \vert_{t=0} = \theta_0$.  

 \item[(iii)] For almost every $t \in [0,T]$ we have the identity
\begin{equation}\label{weak_form_rigid}
 \br{\dt \theta,v}_\ast +  B^\beta(\theta-\theta_{eq}^\beta,v) + \int_\Omega \kappa \nab (\theta-\theta_{eq}^\beta) \cdot \nab v + u \cdot \nab (\theta-\theta_{eq}^\beta) v    =- \int_\Omega  u \cdot \nab \theta_{eq}^\beta v
\end{equation}
for all $v \in H^1_\beta(\Omega)$.  Here $\br{\cdot,\cdot}_\ast$ denotes the   pairing between $H^1_\beta(\Omega)$ and its dual, and  $B^\beta : H^1_\beta(\Omega) \times  H^1_\beta(\Omega) \to \R$ is the bilinear form defined by
\begin{equation}
B^\beta(w,v) =
\begin{cases}
\int_{\Sigma_+}\beta_+ w v + \int_{\Sigma_-} \beta_- w v &\text{if } \beta \in [0,\infty)^2 \\
 \int_{\Sigma_-} \beta_- w v &\text{if } \beta_+ = \infty, \beta_- \in [0,\infty)  \\
 \int_{\Sigma_+}\beta_+ w v &\text{if } \beta_- = \infty, \beta_+ \in [0,\infty)  \\
 0 &\text{if } \beta_+ = \beta_-= \infty.
\end{cases} 
\end{equation}

\end{itemize}
 
\end{dfn}

\begin{Remark}
 The justification of the weak formulation \eqref{weak_form_rigid} is standard: if $\theta$ is a smooth solution to \eqref{passive_rigid2} then \eqref{weak_form_rigid} may be derived through integration by parts and use of the boundary conditions satisfied by $\theta_{eq}^\beta$.  In this case we have
\begin{equation}
 \br{\dt \theta,v}_\ast = \int_\Omega \dt \theta v.
\end{equation}
\end{Remark}

\begin{Remark}
Standard arguments and the inclusions of $(i)$ imply that  
\begin{equation}\label{weak_data}
\theta - \theta_{eq}^\beta \in C^0([0,T]; L^2(\Omega)). 
\end{equation}
This is the sense in which $(ii)$ is required to hold.
\end{Remark}

\begin{Remark}
 The fact that $B^\beta$ is well-defined on $H^1_\beta(\Omega) \times H^1_\beta(\Omega)$ follows from the standard trace theory.
\end{Remark}

\begin{Remark}
We require the  inclusion \eqref{rigid_assump_1} in order to make the integrals 
\begin{equation}
\int_\Omega u \cdot \nab (\theta-\theta_{eq}^\beta) v    \text{ and } \int_\Omega  u \cdot \nab \theta_{eq}^\beta v
\end{equation}
well-defined for $\theta - \theta_{eq}^\beta, v \in H^1_\beta(\Omega)$.  Notice, though, that our definition of weak solution does not require that $\diverge{u}=0$.  We will invoke this condition only in a global setting. 
\end{Remark}

We may readily prove the following local well-posedness result for \eqref{passive_rigid2}.  The proof, which we omit for the sake of brevity, is a standard application of a Galerkin method based on the same energy estimates that we will use for our global existence and decay results.  We refer to Section \ref{sec_energy} for the energy estimates.

\begin{Proposition}\label{rigid_lwp}
Fix $T \in (0,\infty)$.  Let $u$ satisfy \eqref{rigid_assump_1} and let $\theta_0 \in L^2(\Omega)$.   Then there exists a unique weak solution $\theta: [0,T]\times \Omega \to \R$ to \eqref{passive_rigid2}.  
\end{Proposition}

Now we turn to the moving boundary problem.  We will state the definition first and then justify the requirements afterward.  Recall that $\mathcal{A}$, $J$, $K$, and $\n$ are all defined in \eqref{A_def}, \eqref{ABJ_def}, and \eqref{N_def}.

\begin{dfn}\label{weak_moving_def}
Let $T \in (0,\infty]$.  Suppose that $u: [0,T] \times \Omega \to \R^3$ and $\eta : [0,T] \times \Gamma \to \R$ are given and satisfy the inclusions
\begin{equation}\label{weak_geo_inclusion}
\begin{split}
  u & \in L^\infty([0,T];L^\infty(\Omega)) \cap L^2([0,T];L^2(\Omega)), \\
 \eta &\in  L^\infty([0,T];H^{5/2}(\Gamma)) \cap L^2([0,T];H^{3/2}(\Gamma)),   \\
\dt \eta &\in  L^\infty([0,T];H^{3/2}(\Gamma)) \cap    L^2([0,T];L^2(\Gamma)) 
\end{split}
\end{equation}
as well as the estimate
\begin{equation}\label{weak_geo_bound}
 \norm{\eta}_{L^\infty H^{5/2}} \le \delta,
\end{equation}
where $\delta \in (0,1)$ is given in Lemma \ref{eta_small}.  Let $\theta_0 \in L^2(\Omega)$. 

We say that a a function $\theta: [0,T] \times \Omega \to \R$ is a weak solution to \eqref{geometric1} if the following three conditions are satisfied.
\begin{itemize}
 \item[(i)] We have the inclusions 
\begin{equation}
 \theta - \theta_{eq}^\beta \in L^2([0,T]; H^1_\beta(\Omega)) \text{ and } \dt \theta \in L^2([0,T]; (H^1_\beta(\Omega))^\ast ).
\end{equation}

 \item[(ii)] $(\theta - \theta_{eq}^\beta) \vert_{t=0} = \theta_0$.   
 
 \item[(iii)] For almost every $t \in [0,T]$ we have the identity
\begin{multline}\label{weak_form_geo}
 \br{\dt \theta,J v}_\ast +  C_t^\beta(\theta-\theta_{eq}^\beta,v) + \int_\Omega \kappa J \naba (\theta-\theta_{eq}^\beta) \cdot \naba v + u \cdot \naba (\theta-\theta_{eq}^\beta) v J  - \dt \bar{\eta} \tilde{d} \p_3 (\theta-\theta_{eq}^\beta) v\\
  = F_t^\beta(v) + \int_\Omega \left( \dt \bar{\eta} \tilde{d} K \p_3 \theta_{eq}^\beta  -u_j \a_{jk} \p_k\theta_{eq}^\beta +\kappa \a_{jl} \p_l \left(\a_{jk} \p_k \theta_{eq}^\beta \right) \right) v J
\end{multline}
for all $v \in H^1_\beta(\Omega)$.  Here $\br{\cdot,\cdot}_\ast$ denotes the   pairing between $H^1_\beta(\Omega)$ and its dual,   $C_t^\beta : H^1_\beta(\Omega) \times  H^1_\beta(\Omega) \to \R$ is the $t-$dependent bilinear form defined by
\begin{equation}
C_t^\beta(w,v) =
\begin{cases}
\int_{\Sigma_+}\beta_+ w v \abs{\n} + \int_{\Sigma_-} \beta_- w v K &\text{if } \beta \in [0,\infty)^2 \\
 \int_{\Sigma_-} \beta_- w v K &\text{if } \beta_+ = \infty, \beta_- \in [0,\infty)  \\
 \int_{\Sigma_+}\beta_+ w v \abs{\n} &\text{if } \beta_- = \infty, \beta_+ \in [0,\infty)  \\
 0 &\text{if } \beta_+ = \beta_-= \infty,
\end{cases} 
\end{equation}
and $F_t^\beta : H^1_\beta(\Omega) \to \R$ is the $t-$dependent force term given by
\begin{equation}\label{force_def}
F_t^\beta(v) =
\begin{cases}
\int_{\Sigma_+} \beta_+ (\bar{\theta} - \theta_{eq}^\beta) \abs{\n} (1- K \abs{\n}) v&\text{if } \beta_+ \in [0,\infty)  \\
 0 &\text{otherwise }.
\end{cases} 
\end{equation}

\end{itemize}
 
\end{dfn}

Some remarks are in order.

\begin{Remark}
 The weak formulation \eqref{weak_form_geo} is justified in the same way as \eqref{weak_form_rigid}.  The only difference is that the computations in the integration by parts are somewhat more involved.  For the sake of readability we have moved the computation to the appendix: see Lemma \ref{mov_smooth_ident}.
\end{Remark}

\begin{Remark}
The influence of the moving boundary is manifest in \eqref{weak_form_geo} through the appearance of the terms $J,$ $\a$, etc, all of which depend on $\eta$, the free surface function.  
\end{Remark}

\begin{Remark}
 The inclusions  \eqref{weak_geo_inclusion} and the bound \eqref{weak_geo_bound} allow us to employ Lemmas \ref{eta_poisson}, \ref{eta_small}, and \ref{nonlin_ests}.  Together these guarantee that all of the integrals and the dual-pairing in \eqref{weak_form_geo} are well-defined.  
\end{Remark}

\begin{Remark}
In the weak solution definition we make no assumptions about $u$ and $\eta$ satisfying any of the equations of \eqref{geometric}.  We will need some of these only in the global theory.
\end{Remark}

We may again readily deduce the existence of local weak solutions to \eqref{geometric1} by employing a standard  Galerkin method based on the energy estimates that we will use for our global existence and decay results, and Lemmas \ref{eta_poisson}, \ref{eta_small}, and \ref{nonlin_ests}.  Notice in particular that Lemma \ref{eta_small} provides for the ellipticity condition on $\a$ and for the fact that the boundary integrals are controlled via the usual $H^1(\Omega)$ trace theory.  We again refer to Section \ref{sec_energy} for the energy estimates, and again omit the proof of local existence result.

\begin{Proposition}\label{geo_lwp}
Fix $T \in (0,\infty)$.  Let $u,\eta$ satisfy \eqref{weak_geo_inclusion} and \eqref{weak_geo_bound}, and let $\theta_0 \in L^2(\Omega)$.   Then there exists a unique weak solution $\theta: [0,T]\times \Omega \to \R$ to \eqref{geometric1}.  
\end{Proposition}

%======================
\subsection{Global solutions and equilibration}
%======================

We now state our main results on the existence of global solutions and their decay to equilibrium.  We begin with the case of a rigid domain.

\begin{theorem}[Rigid domain]\label{thm-ri}
Let $\Omega$ be the rigid domain defined by \eqref{omega_def} with either $\Gamma = \R^2$ or $\Gamma = (L_1 \mathbb{T}) \times (L_2 \mathbb{T})$.  Assume that $u:\R_+ \times \Omega \to \R^3$ satisfies 
\begin{equation}
 u \in L^\infty(\R_+; L^\infty(\Omega)) \cap L^2(\R_+;H^1(\Omega)),
\end{equation}
and that for a.e. $t \in \R_+$ $u(t,\cdot)$ satisfies the incompressibility condition 
\begin{equation}\label{ri_01}
\diverge u =0 
\end{equation}
 as well as the boundary conditions
\begin{equation}\label{ri_02} 
u_3 =0 \text{ on } \Sigma_+ \text{ and } \Sigma_-,   
\end{equation}
where $\Sigma_\pm$ are defined by \eqref{sigma_def}.   Let $\theta_0 \in L^2(\Omega)$.  

Then there exists a unique $\theta: \R_+ \times \Omega \to \R$ that satisfies $(\theta - \theta_{eq}^\beta)\vert_{t=0} = \theta_0$ and 
is the weak solution to \eqref{passive_rigid2} in the sense of Definition \ref{weak_rigid_def} with $T= \infty$.  Moreover, $\theta$ obeys the following estimates, where $\mu_\beta >0$ is as defined in Lemma \ref{eigenvalue} when $\beta \neq (0,0)$ and $\bar{\mu}_0 >0$ is as defined in Proposition \ref{rigid_coercive_zero} when $\beta=(0,0)$.
\begin{itemize}
\item[(i)] When $\beta_\pm =0$ (pure Neumann boundary conditions) and $\Gamma = (L_1 \mathbb{T}) \times (L_2 \mathbb{T})$, then we define 
\begin{equation}
 \theta_{avg} = \frac{1}{\abs{\Omega}} \int_\Omega \theta_0.  
\end{equation}
Then 
\begin{equation}\label{avg_cond}
 \frac{1}{\abs{\Omega}} \int_\Omega (\theta(t,\cdot) - \theta_{eq}^\beta) = \theta_{avg} \text{ for all }t\in \R_+.
\end{equation}
Also, regardless of the decay of $u$, $\theta-\theta_{eq}^\beta-\theta_{avg}$ decays to $0$ exponentially fast with the following estimate 
\begin{equation}\label{decay00}
\| \theta(t,\cdot) - \theta_{eq}^\beta  - \theta_{avg} \|_{L^2(\Omega)} \leq \exp\left(-\bar{\mu}_0 t\right) \| \theta_0 - \theta_{avg} \|_{L^2(\Omega)}. 
\end{equation}

\item[(ii)] When $\beta_\pm =0$ and $\Gamma = \R^2$, we have the estimate
\begin{equation}
 \sup_{t \ge 0} \| \theta(t,\cdot) - \theta_{eq}^\beta   \|_{L^2(\Omega)}^2 + \int_0^\infty \ns{\nab (\theta(t,\cdot) - \theta_{eq}^\beta)}_{L^2(\Omega)}dt  \le \frac{3}{2} \ns{\theta_0}_{L^2(\Omega)}. 
\end{equation}

\item[(iii)] When $\beta \in [0,\infty]^2 \backslash \{(0,0)\}$, we further assume the $L^2$ decay of $u$: 
\begin{equation}
\|u(t,\cdot)\|_{L^2(\Omega)} \leq g(t), 
\end{equation}
where $g(t) \to 0$ as $t \to \infty$. Then $\theta -\theta_{eq}^\beta$ decays to $0$ with the following estimate
\begin{equation}\label{decay1}
 \|\theta(t,\cdot) -\theta^\beta_{eq} \|_{L^2(\Omega)} \leq e^{-\mu_\beta t}  \|\theta_0  \|_{L^2(\Omega)}  + \abs{\p_3 \theta_{eq}^\beta}  \int_0^t e^{-\mu_\beta (t-s)} g(s)ds . 
\end{equation}
\end{itemize}
\end{theorem}

Some remarks are in order.

\begin{Remark}
 When $\beta = (0,0)$ and $\Gamma = \R^2$ we fail to actually show that the solutions decay to equilibrium.  This is due to the lack of coercivity induced by the infinite cross-section.  Instead we have bounds on the deviation of $\theta$ from the equilibrium.  

 When $\beta =(0,0)$ and $\Gamma = (L_1 \mathbb{T}) \times (L_2 \mathbb{T})$ we find that $\theta$ converges exponentially fast to the constant $\theta_{eq}^\beta + \theta_{avg}$.  Recall that $\theta_{eq}^\beta$ is an arbitrary constant when $\beta=(0,0)$.  This means that we can view the equilibrium as being determined by the data in the sense that  \eqref{avg_cond} implies 
\begin{equation}
 \theta_{avg} + \theta_{eq}^\beta = \frac{1}{\abs{\Omega}} \int_\Omega \theta(0,\cdot) dx.
\end{equation}

\end{Remark}

\begin{Remark}
 The appearance of $\mu_\beta$  in Proposition \ref{rigid_coercive}  and $\bar{\mu}_0$  in Proposition \ref{rigid_coercive_zero} show that the exponential decay rates are optimal.  Lemma \ref{eigenvalue} also provides some estimates of $\mu_\beta$ in terms of $\beta \in [0,\infty]^2$.
\end{Remark}

\begin{Remark}
In the case of insulating boundary conditions, $\beta =(0,0)$, the decay rate is independent of the velocity field $u$.  The same is true for any $\beta \in [0,\infty]^2$ if $\bar{\theta}=0$ since in this case $\theta_{eq}^\beta =0$.
\end{Remark}

\begin{Remark} 
 If $g(t)$ decays faster than $e^{-\mu_\beta t} $, then the decay rate of $\theta$ is dictated by $e^{-\mu_\beta t}$.  If $g(t) \sim e^{-\mu_\beta t}$, then 
\begin{equation}
 \|\theta(t,\cdot) -\theta^\beta_{eq} \|_{L^2(\Omega)} \leq  C e^{-\mu_\beta t}\left( 1+t \right)
\end{equation}
for some constant $C$.   If $g(t)$ decays slower than $e^{-\mu_\beta t}$, the decay is dominated by $g(t)$.  
\end{Remark}

In the case of moving boundary, we have the following result.

\begin{theorem}[Moving boundary]\label{thm-mo}
Let $\Omega$ be  defined by \eqref{omega_def} with either $\Gamma = \R^2$ or $\Gamma = (L_1 \mathbb{T}) \times (L_2 \mathbb{T})$. 
Assume that $u:\R_+ \times \Omega \to \R^3$ and $\eta: \R_+ \times \Gamma \to \R$ satisfy \eqref{weak_geo_inclusion} and \eqref{weak_geo_bound} and that 
\begin{equation}
 u \in L^2(\R_+;H^1(\Omega)).
\end{equation}
Assume also that for a.e. $t \in \R_+$   $u$ and $\eta$ satisfy the 
 $\mathcal{A}-$incompressibility condition (defined below \eqref{geometric})
\begin{equation}\label{mo_01}
\diva u = 0 \text{ in }\Omega, 
\end{equation}
the kinematic boundary condition 
\begin{equation}
 \dt \eta = u \cdot \n \text{ on }\Sigma_+,
\end{equation}
the boundary condition 
\begin{equation}\label{mo_02}
u_3 =0 \text{ on } \Sigma_-, 
\end{equation}
and the bound 
\begin{equation}\label{mo_03}
 1/c_0\leq J\leq c_0
\end{equation}
for some constant $c_0>1$, where $J$ is defined by \eqref{ABJ_def}.   Let $\theta_0 \in L^2(\Omega)$.  

Then there exists a unique $\theta: \R_+ \times \Omega \to \R$ that satisfies $(\theta - \theta_{eq}^\beta)\vert_{t=0} = \theta_0$ and is the weak solution to \eqref{geometric1} in the sense of Definition \ref{weak_moving_def} with $T= \infty$.  Moreover, $\theta$ obeys the following estimates, where $\mu_\beta >0$ is as defined in Lemma \ref{eigenvalue} when $\beta \neq (0,0)$ and $\bar{\mu}_0 >0$ is as defined in Proposition \ref{rigid_coercive_zero} when $\beta=(0,0)$.
\begin{itemize}
\item[(i)]  When $\beta_\pm =0$ (pure Neumann boundary conditions) and $\Gamma = (L_1 \mathbb{T}) \times (L_2 \mathbb{T})$, then we define 
\begin{equation}
 \theta_{avg} = \left(\int_\Omega J(0,\cdot)dx \right)^{-1}  \int_\Omega \theta_0 J(0,\cdot)dx.  
\end{equation}
Then 
\begin{equation} 
 \int_\Omega (\theta(t,\cdot) - \theta_{eq}^\beta - \theta_{avg}) J(t,\cdot) dx = 0 \text{ for all }t \in \R_+.
\end{equation}
Also, if we additionally assume that 
\begin{equation}\label{mo_04}
 \frac{1}{c_1} I \le  J \a^T \a \le c_1 I \text{ for some } c_1 >0,
\end{equation}
then regardless of the decay of $u$, $\theta-\theta_{eq}^\beta-\theta_{avg}$ decays to $0$ exponentially fast with the following estimate:
\begin{equation}\label{decay0}
\| (\theta(t,\cdot) - \theta_{eq}^\beta-\theta_{avg}) \sqrt{J(t,\cdot)} \|_{L^2(\Omega)} \leq \exp\left( -\frac{\bar{\mu}_0}{c_0c_1} t\right) \| (\theta_0 -\theta_{avg})\sqrt{J(0,\cdot)}\|_{L^2(\Omega)}. 
\end{equation}

\item[(ii)] When $\beta_\pm =0$ and $\Gamma = \R^2$, we have the estimate
\begin{multline}\label{decay0_2}
 \sup_{t \ge 0} \| (\theta(t,\cdot) - \theta_{eq}^\beta)\sqrt{J(t,\cdot)}   \|_{L^2(\Omega)}^2 + \int_0^\infty \ns{\sqrt{J(t,\cdot)} \naba (\theta(t,\cdot) - \theta_{eq}^\beta)}_{L^2(\Omega)} \\
 \le \frac{3}{2} \ns{(\theta_0 -\theta_{avg})\sqrt{J(0,\cdot)}}_{L^2(\Omega)}. 
\end{multline}

\item[(iii)] When $\beta_+=\infty$  we assume the decay result: 
\begin{equation}\label{decay_assump_2}
\|u(t,\cdot)\|_{L^2(\Omega)} + \|\partial_t \eta(t,\cdot)\|_{L^2(\Gamma)} + \|\nab_\ast \eta(t,\cdot)\|_{H^{1/2}(\Gamma)} \leq  g(t) 
\end{equation}
where $g(t) \to 0$ as $t \to \infty$ (here $\nab_\ast$ is the horizontal gradient on $\Gamma$).   Then $\theta -\theta^\beta_{eq}$ decays with the following estimate
\begin{equation}\label{decay2}
\begin{split}
 \|(\theta(t,\cdot) -\theta^\beta_{eq}) \sqrt{J(t,\cdot)} \|_{L^2(\Omega)} \leq & \exp\left(-\frac{\mu_\beta}{c_0^2} t\right)  \|\theta_0\sqrt{J(0,\cdot)} \|_{L^2(\Omega)}  \\&+ C  \abs{\p_3 \theta_{eq}^\beta} \int_0^t \exp\left(-\frac{\mu_\beta}{c_0^2} (t-s)\right) g(s)ds 
 \end{split}
\end{equation}
where $C>0$ is a universal constant. 

\item[(iv)] When $\beta \in [0,\infty) \times [0,\infty] \backslash \{(0,0)\}$, we assume that
\begin{equation}\label{decay_assump_3}
\|u(t,\cdot)\|_{L^2(\Omega)} + \|\partial_t \eta(t,\cdot)\|_{L^2(\Gamma)} + \| \eta(t,\cdot)\|_{H^{3/2}(\Gamma)} \leq  h(t) 
\end{equation}
where $h(t) \to 0$ as $t \to \infty$.  Then $\theta$ decays to $\theta^\beta_{eq}$ with the following estimate
\begin{multline}\label{decay3}
 \|(\theta(t,\cdot) -\theta^\beta_{eq}) \sqrt{J(t,\cdot)} \|^2_{L^2(\Omega)} \leq  \exp\left(-\frac{\mu_\beta}{2c_0^2} t\right)  \|\theta_0 \sqrt{J(0,\cdot)} \|^2_{L^2(\Omega)}  \\ 
+ C \left( \abs{\p_3 \theta_{eq}^\beta}^2  \frac{c_0^2}{\mu_\beta}  + \frac{\beta_+  }{2}\abs{\bar{\theta}-\theta_{eq}^\beta}^2 \right)   \int_0^t \exp\left( -\frac{\mu_\beta}{2 c_0^2} (t-s)\right)  h^2(s) ds 
\end{multline}
where $C>0$ is a universal constant. 
\end{itemize}
\end{theorem}

\begin{Remark}
 Again we fail to prove decay when $\beta =(0,0)$ and $\Gamma = \R^2$ because of the failure of the dissipation to be $L^2$-coercive in this case. 
\end{Remark}

\begin{Remark}
The conditions \eqref{mo_03} and \eqref{mo_04} are satisfied because of  Lemma \ref{eta_small}.  The bound \eqref{mo_03} shows that $\theta(t,\cdot) - \theta_{eq}^\beta$ actually decays except when $\beta=(0,0)$ and $\Gamma = \R^2$..
\end{Remark}

\begin{Remark} In the moving boundary case, the decay rate is not as explicit as the rigid case. We need the decay  of not only velocity $u$ but also moving boundary $\eta$ in order to derive the decay rate of $\theta$ for $\beta>0$. 
\end{Remark}

\begin{Remark}
 Again if $\bar{\theta}=0$ then the decay is not influenced by the asymptotics of $u$.  The asymptotics of $\eta$, and hence $\Omega(t)$, do play a role in the bounds on $J$, which contribute the $c_0$ terms to the decay rate.
\end{Remark}

\begin{Remark}
The reason we include $\norm{\nab_\ast \eta}_{H^{1/2}(\Gamma)}$ in \eqref{decay_assump_2} but $\norm{ \eta}_{H^{3/2}(\Gamma)}$ in \eqref{decay_assump_3} is that in case $(iii)$ we also need $\norm{\eta}_{H^1(\Gamma)}$ to decay, which in particular requires the $L^2$ decay of $\eta$.  This is due to the appearance of the term $1-K \abs{\n}$ in $F_t^\beta$.

This term also causes the difference in the structure of the estimates \eqref{decay3} and \eqref{decay0}, \eqref{decay2}.  The former estimates the decay of the square of the $L^2$ norm while the latter two estimate the decay of the $L^2$ norm.
\end{Remark}

We conclude our discussion of the moving boundary problem with a couple concrete of examples.  In particular, we consider cases $(iii)$ and $(iv)$ under the extra assumption that $u$ and $\eta$ satisfy the system \eqref{geometric}.  Initially suppose that $\sigma >0$ in \eqref{geometric}, which corresponds to surface tension on the moving interface.  When $\Gamma = \R^2$ the work of Beale-Nishida \cite{BN} showed that 
\begin{equation}
 g(t) \le \frac{C}{1+t} \text{ and } h(t) \le \frac{C}{\sqrt{1+t}}
\end{equation}
with the slower rate of decay for $h$ determined by the slow decay of $\norm{\eta(t,\cdot)}_{L^2}$.  When $\Gamma = (L_1 \mathbb{T}) \times (L_2 \mathbb{T})$ the work of Nishida-Teramoto-Yoshihara \cite{nishida_1} showed that 
\begin{equation}
 g(t) \le C e^{-\gamma t}  \text{ and } h(t) \le C e^{-\gamma t}
\end{equation}
for some $\gamma >0$.  Thus we see that the faster decay in the periodic case also leads to faster decay of the passive scalar.

Now consider the case $\sigma =0$.  When $\Gamma = \R^2$ the work of Guo-Tice \cite{GT_inf} showed that
\begin{equation}
 g(t) \le \frac{C}{(1+t)^{(1+\lambda)/2}} \text{ and } h(t) \le \frac{C}{(1+t)^{\lambda/2}},
\end{equation}
where $\lambda \in (0,1)$ is a measure of the negative Sobolev regularity of the initial data.  When $\Gamma = (L_1 \mathbb{T}) \times (L_2 \mathbb{T})$, Guo-Tice showed in \cite{GT_per} that for any $m \ge 4$ there exists a smallness condition on the data for \eqref{geometric} that guarantee that 
\begin{equation}
 g(t) \le \frac{C}{(1+t)^{m}} \text{ and } h(t) \le \frac{C}{(1+t)^{m}}.
\end{equation}
Again we find that the periodicity leads to faster decay rates.

%======================
%======================
\section{$L^2$-energy estimates}\label{sec_energy}
%======================
%======================

In this section we present the basic $L^2$ energy estimates.

%======================
\subsection{Rigid boundary}
%======================

We now derive the $L^2$ energy estimate in the case of a static rigid domain.   In order to state these, we must first define the dissipation functional in terms of $\beta.$   Recall the spaces $H^1_\beta(\Omega)$ defined in \eqref{sobolev_beta}.  We define $\mathcal{D}_\beta : H^1_\beta(\Omega) \to \mathbb{R}$ via
\begin{equation}\label{Dthe}
 \mathcal{D}_\beta[\varphi] = 
\begin{cases}
  \int_\Omega \kappa \abs{\nab \varphi}^2 + \beta_+ \int_{\Sigma_+} \abs{\varphi}^2 + \beta_- \int_{\Sigma_-} \abs{\varphi}^2 &\text{if } \beta \in [0,\infty)^2 \\
  \int_\Omega \kappa\abs{\nab \varphi}^2 + \beta_- \int_{\Sigma_-} \abs{\varphi}^2 &\text{if } \beta_+ = \infty, \beta_- \in [0,\infty) \\
   \int_\Omega \kappa\abs{\nab \varphi}^2 + \beta_+ \int_{\Sigma_+} \abs{\varphi}^2    &\text{if } \beta_- = \infty, \beta_+ \in [0,\infty)  \\
   \int_\Omega \kappa \abs{\nab \varphi}^2     &\text{if } \beta_+ = \beta_-= \infty.
\end{cases}
\end{equation}

\begin{lemma}\label{lem1}
Assume that $u$ satisfies the assumptions of Theorem \ref{thm-ri},  and suppose $\theta$ is a weak solution to \eqref{passive_rigid2}. Then we have 
\begin{equation}\label{ene1}
\frac12\frac{d}{dt} \int_{\Omega} |\theta-\theta^\beta_{eq}|^2 dx  + \mathcal{D}_\beta[\theta - \theta_{eq}^\beta] = -  \p_3 \theta_{eq}^\beta  \int_{\Omega} (\theta-\theta^\beta_{eq})\,  u_3 dx
\end{equation}
where $\mathcal{D}_\theta$ is given by \eqref{Dthe}.
\end{lemma}

\begin{proof}

We may use $\theta - \theta_{eq}^\beta$ as a test function in \eqref{weak_form_rigid} to deduce that
\begin{equation}
 I + II = III,
\end{equation}
where 
\begin{equation}
 I = \br{\dt \theta,\theta - \theta_{eq}^\beta}_\ast,
\end{equation}
\begin{equation}
 II = B^\beta(\theta - \theta_{eq}^\beta,\theta - \theta_{eq}^\beta) + \int_\Omega \kappa \abs{\nab (\theta - \theta_{eq}^\beta)}^2 + u \cdot \nab(\theta - \theta_{eq}^\beta) (\theta - \theta_{eq}^\beta),
\end{equation}
and
\begin{equation}
 III = - \int_\Omega u \cdot \nab \theta_{eq}^\beta (\theta - \theta_{eq}^\beta).
\end{equation}

A standard computation shows that
\begin{equation}
 I = \hal \frac{d}{dt} \int_{\Omega} |\theta-\theta^\beta_{eq}|^2 dx.
\end{equation}
We may use the fact that $\diverge u=0$ in $\Omega$ and $u_3 =0$ on $\p \Sigma_\pm$ to compute 
\begin{equation}
 \int_\Omega  u \cdot \nab(\theta - \theta_{eq}^\beta) (\theta - \theta_{eq}^\beta) = \int_\Omega  u \cdot \nab \frac{(\theta - \theta_{eq}^\beta)^2}{2}  =0.
\end{equation}
Then 
\begin{equation}
 II = \mathcal{D}_\beta[\theta - \theta_{eq}^\beta].
\end{equation}
Finally, we compute 
\begin{equation}
 III =  -  \p_3 \theta_{eq}^\beta  \int_{\Omega} (\theta-\theta^\beta_{eq})\,  u_3 dx.
\end{equation}
The  equality  \eqref{ene1} then follows by combining these.

\end{proof}

We remark that the right-hand-side of \eqref{ene1} does not have a definite sign in general and it includes not only the temperature fluctuation but also the vertical velocity $u_3$.  

In the case of the Neumann boundary condition, when $\beta=(0,0)$,  the equilibrium is given by  $\theta_{eq}^\beta=C$.  In this case  the right-hand-side of \eqref{ene1} is zero, and hence we obtain the equality 
\begin{equation}\label{ene1_n}
\frac12\frac{d}{dt} \int_{\Omega} |\theta - \theta_{eq}^\beta|^2 dx  + \kappa   \int_{\Omega} |\nabla \theta|^2 dx =0 . 
\end{equation}
In order to deduce decay information we need a slightly different version of this equality.

\begin{lemma}\label{lem1_zero}
Let $\beta = (0,0)$ and $\Gamma = (L_1 \mathbb{T}) \times (L_2 \mathbb{T})$.  Assume that $u$ satisfies the assumptions of Theorem \ref{thm-ri},  and suppose $\theta$ is a weak solution to \eqref{passive_rigid2}.   Then we have 
\begin{equation}\label{l1z_01}
\frac12\frac{d}{dt} \int_{\Omega} |\theta-\theta^\beta_{eq} -C |^2 dx  + \mathcal{D}_\beta[\theta - \theta_{eq}^\beta -C] = 0
\end{equation}
for any $C \in \R$.  Also, 
\begin{equation}\label{l1z_02}
  \int_{\Omega} (\theta(t,\cdot)-\theta^\beta_{eq})dx = \int_\Omega \theta_0 dx
\end{equation}
 for all $t \in [0,T]$.
\end{lemma}
\begin{proof}
 The proof  of \eqref{l1z_01} is the same as for Lemma \ref{lem1}, except that we use $\theta - \theta_{eq}^\beta -C \in H^1_\beta(\Omega) = H^1(\Omega)$ as the test function.  To prove \eqref{l1z_02} we use $1 \in H^1(\Omega)$ as the test function and argue similarly to deduce that 
\begin{equation}
 \frac{d}{dt}   \int_{\Omega} (\theta(t,\cdot)-\theta^\beta_{eq})dx  =0,
\end{equation}
which yields  \eqref{l1z_02} upon integrating in time.

\end{proof}

%======================
\subsection{Moving boundary}
%======================

We now seek to derive the $L^2$ energy estimate in the case of a moving boundary.  In this case  we define the time-dependent dissipation functional   $\mathcal{M}^t_\beta : H^1_\beta(\Omega) \to \mathbb{R}$ via
\begin{equation}\label{Dthe2}
 \mathcal{M}^t_\beta[\varphi] = 
\begin{cases}
  \int_\Omega \kappa J \abs{\naba \varphi}^2 + \beta_+ \int_{\Sigma_+} \abs{\varphi}^2 \abs{\n} + \beta_- \int_{\Sigma_-} \abs{\varphi}^2 K &\text{if } \beta \in [0,\infty)^2 \\
  \int_\Omega \kappa J \abs{\naba \varphi}^2 + \beta_- \int_{\Sigma_-} \abs{\varphi}^2 K &\text{if } \beta_+ = \infty, \beta_- \in [0,\infty) \\
   \int_\Omega \kappa J \abs{\naba \varphi}^2 + \beta_+ \int_{\Sigma_+} \abs{\varphi}^2 \abs{\n}    &\text{if } \beta_- = \infty, \beta_+ \in [0,\infty)  \\
   \int_\Omega \kappa J \abs{\naba \varphi}^2     &\text{if } \beta_+ = \beta_-= \infty.
\end{cases}
\end{equation}
Here we have written $\mathcal{M}^t_\beta$ to emphasize the dependence on time $t$: $J$, $K$, $\a$, and $\n$ are all understood to be evaluated at time $t$ in \eqref{Dthe2}

We can derive the $L^2$ energy estimate.  We will employ Lemma \ref{geometric_ids} for many of the calculations.
 
\begin{lemma}\label{geo_energy}
Suppose $u$ and $\eta$ satisfy the assumptions of Theorem \ref{thm-mo} and that  $\theta$ is a weak solution to \eqref{geometric1}.  Then 
\begin{multline}\label{ene2}
\hal \frac{d}{dt} \int_\Omega  | \theta  -\theta_{eq}^\beta|^2J dx + \mathcal{M}_\beta^t[\theta - \theta_{eq}^\beta ] = F_t^\beta(\theta- \theta_{eq}^\beta) \\
+  \p_3 \theta_{eq}^\beta
 \int_\Omega     (\theta -\theta_{eq}^\beta) \left\{   \dt \bar{\eta} \tilde{d}  -u_j J \a_{j3}  +\kappa J \a_{jl} \p_l\a_{j3}  \right\}  dx.
\end{multline}
where  $\mathcal{M}^t_\beta $ is given by \eqref{Dthe2} and $F_t^\beta$ is given by \eqref{force_def}.

\end{lemma}

\begin{proof}
We use $\theta - \theta_{eq}^\beta \in H^1_\beta(\Omega)$ as a test function in \eqref{weak_form_geo} to derive the equality
\begin{equation}
I + II = III, 
\end{equation}
where 
\begin{align}
I&=  \br{\dt \theta,J (\theta - \theta_{eq}^\beta)}_\ast  + \int_\Omega     J(\theta -\theta_{eq}^\beta) \left\{  - \dt \bar{\eta} \tilde{d} K \p_3 
  \left( \theta  -\theta_{eq}^\beta \right)+ u_j \a_{jk} \p_k\left( \theta  -\theta_{eq}^\beta \right) \right\}  dx,\label{I} \\
II&= C_t^\beta(\theta-\theta_{eq}^\beta,\theta-\theta_{eq}^\beta) + \int_\Omega \kappa J \abs{\naba (\theta-\theta_{eq}^\beta)}^2 dx, \label{II} \\
III&= F_t^\beta(\theta - \theta_{eq}^\beta) +\int_\Omega     J(\theta -\theta_{eq}^\beta) \left\{   \dt \bar{\eta} \tilde{d} K \p_3 \theta_{eq}^\beta  -u_j \a_{jk} \p_k\theta_{eq}^\beta +\kappa \a_{jl} \p_l \left(\a_{jk} \p_k \theta_{eq}^\beta \right) \right\}  dx. \label{III}
\end{align}

Employing \eqref{id1} and the identity $JK=1$, we rewrite $I$ as 
\begin{equation}
\begin{split}
I=I_1+I_2& =: \frac12\frac{d}{dt} \int_\Omega J | \theta  -\theta_{eq}^\beta|^2 dx\\
&  +\int_\Omega -\frac{ \p_t J  | \theta  -\theta_{eq}^\beta|^2 }{2} -  \dt \bar{\eta} \tilde{d}\p_3 \frac{  | \theta  -\theta_{eq}^\beta|^2  }{2}   + 
u_j  \p_k\left[ J\a_{jk}   \frac{|\theta  -\theta_{eq}^\beta|^2}{2} \right]  dx. 
\end{split}
\end{equation}
Identity \eqref{id4} shows that $u_j J \mathcal{A}_{jk} e_3 \cdot e_k = u_3 =0$ on $\Sigma_d$, and we know that $\tilde{d} =0$ on $\Sigma_d$, so we may integrate by parts to see that
\begin{equation}
\begin{split}
I_2&= \int_\Omega -\frac{ \p_t J  | \theta  -\theta_{eq}^\beta|^2 }{2} -  \dt \bar{\eta} \tilde{d}\p_3 \frac{  | \theta  -\theta_{eq}^\beta|^2  }{2}   + 
u_j  \p_k\left[ J\a_{jk}   \frac{|\theta  -\theta_{eq}^\beta|^2}{2} \right]  dx \\
&=   \int_\Omega -\frac{ \p_t J  | \theta  -\theta_{eq}^\beta|^2 }{2} +\p_3( \dt \bar{\eta} \tilde{d}) \frac{  | \theta  -\theta_{eq}^\beta|^2  }{2} dx -  \int_\Omega \p_k u_j J\a_{jk}   \frac{|\theta  -\theta_{eq}^\beta|^2}{2}   dx \\
&\quad+\int_\Sigma -\p_t\eta  \frac{|\theta  -\theta_{eq}^\beta|^2}{2}   +  u_j  J\a_{jk} e_3\cdot e_k  \frac{|\theta  -\theta_{eq}^\beta|^2}{2} d\sigma_x . 
\end{split}
\end{equation}
The incompressibility condition \eqref{mo_01} is equivalent to  $ \a_{jk} \p_k u_j =0 $, and an easy computation shows that $\dt J = \p_3(\dt \bar{\eta}d)$.   Since $\dt \eta= u\cdot \n$ on $\Sigma$, we may use  \eqref{id2} and \eqref{id3} to  deduce that $I_2=0$.  Hence, 
\begin{equation}\label{I2}
I= \frac12\frac{d}{dt} \int_\Omega  | \theta  -\theta_{eq}^\beta|^2J dx. 
\end{equation}

We simply rewrite $II$ in \eqref{II} as
\begin{equation}\label{II2}
 II = \mathcal{M}^t_\beta[\theta  -\theta_{eq}^\beta]. 
\end{equation}

It remains to handle $III$ in \eqref{III}.  Since $\p_k\theta_{eq}^\beta=\delta_{3k}\p_k\theta_{eq}^\beta$ and  $\p_3\theta_{eq}^\beta$ is constant, we can rewrite 
 \begin{equation}
 \begin{split}\label{III2}
 III  &= F_t^\beta(\theta- \theta_{eq}^\beta) + \int_\Omega     J(\theta -\theta_{eq}^\beta) \left\{   \dt \bar{\eta} \tilde{d} K \p_3 \theta_{eq}^\beta  -u_j \a_{jk} \p_k\theta_{eq}^\beta +\kappa \a_{jl} \p_l \left(\a_{jk} \p_k \theta_{eq}^\beta \right) \right\}  dx \\
 =&  F_t^\beta(\theta- \theta_{eq}^\beta) + \p_3 \theta_{eq}^\beta 
 \int_\Omega     (\theta -\theta_{eq}^\beta) \left\{   \dt \bar{\eta} \tilde{d}  -u_j J \a_{j3}  +\kappa J \a_{jl} \p_l\a_{j3}  \right\}  dx.
 \end{split}
 \end{equation}
 Combining \eqref{I2}, \eqref{II2}, \eqref{III2} then yields the desired result. 
\end{proof}

As in the case of rigid boundary, if the insulating boundary condition is imposed ($\beta=(0,0)$), the $L^2$ energy estimates \eqref{ene2} reduce to 
\begin{equation}
\begin{split}\label{ene2_n}
\frac12\frac{d}{dt} \int_{\Omega} |\theta - \theta_{eq}^\beta|^2 J dx  + \kappa \int_{\Omega} | \naba \theta |^2 J dx=0 . 
\end{split}
\end{equation}
Again we need a slightly different version of this estimate for it to be useful in decay analysis.

\begin{lemma}\label{geo_energy_zero}
Let $\beta =(0,0)$ and $\Gamma = (L_1 \mathbb{T}) \times (L_2 \mathbb{T})$.  Suppose $u$ and $\eta$ satisfy the assumptions of Theorem \ref{thm-mo} and that  $\theta$ is a weak solution to \eqref{geometric1}.  Then for any constant $C \in \R$, 
\begin{equation}\label{gez_01}
\hal \frac{d}{dt} \int_\Omega  | \theta  -\theta_{eq}^\beta -C |^2J dx + \mathcal{M}_\beta^t[\theta - \theta_{eq}^\beta -C ] = 0
\end{equation}
where  $\mathcal{M}^t_\beta $ is given by \eqref{Dthe2}.  Also, if we write 
\begin{equation}\label{gez_02}
 \theta_{avg} = \left(\int_{\Omega} J(0,\cdot) dx \right)^{-1} \int_\Omega \theta_0 J(0,\cdot) dx,
\end{equation}
then 
\begin{equation}\label{gez_03}
 \int_\Omega (\theta(t,\cdot) - \theta_{eq}^\beta - \theta_{avg}) J(t,\cdot) dx = 0
\end{equation}
for all $t\in [0,T]$.
\end{lemma}
\begin{proof}
 To prove \eqref{gez_01} we use $\theta(t,\cdot) - \theta_{eq}^\beta - C$ as a test function and argue as in the proof of Lemma \ref{geo_energy}.  To prove \eqref{gez_03} we argue similarly, using $1$ as a test function to derive the equality
\begin{equation}
 \frac{d}{dt} \int_\Omega J (\theta(t,\cdot) - \theta_{eq}^\beta) = 0.
\end{equation}
We also compute, using Lemma \ref{geometric_ids},
\begin{multline}
 \frac{d}{dt} \int_\Omega J(t,\cdot) = \int_\Omega \dt J = \int_\Omega \p_3(\dt \bar{\eta} \tilde{d}) = \int_{\Sigma_+} \dt \eta 
= \int_{\Sigma_+} u \cdot \n \\
= \int_{\Sigma_+} u_j J \a_{j3} = \int_\Omega \diva u J=0.
\end{multline}
Then \eqref{gez_03} follows by integrating these identities and using the definition \eqref{gez_02}.

\end{proof}

%======================
\section{Coercivity of the dissipation}
%======================

In order for the $L^2$ energy estimates of the previous section to give rise to decay results, we need for the dissipation terms $\mathcal{D}_\beta$, defined by  \eqref{Dthe} for rigid boundaries, and $\mathcal{M}_\beta^t$, defined by \eqref{Dthe2} for moving boundaries, to be $L^2$-coercive.  In this section we pursue the proof of this coercivity estimate, which is essentially a Poincar\'e inequality.

%======================
\subsection{One-dimensional analysis}
%======================

When $\beta \neq (0,0)$, our coercivity estimates will be based on a corresponding coercivity estimate in the vertical direction.  Our aim now is to prove this estimate.

Let $H^1_\beta((-d,0))$ denote the space 
\begin{equation}
H^1_\beta((-d,0))=
 \begin{cases}
  H^1((-d,0)) &\text{if } \beta \in [0,\infty)^2 \\
  \{\zeta \in H^1((-d,0)) \;\vert\; \zeta(0) =0  \}  &\text{if } \beta_+ = \infty, \beta_- \in [0,\infty)  \\
  \{\zeta \in H^1((-d,0)) \;\vert\; \zeta(-d) =0  \}  &\text{if } \beta_- = \infty, \beta_+ \in [0,\infty)  \\
  \{\zeta \in H^1((-d,0)) \;\vert\; \zeta(0) = \zeta(-d) =0  \}  &\text{if } \beta_+ = \beta_-= \infty.
 \end{cases}
\end{equation}
We define $\mathfrak{D}_\beta : H^1_\beta((-d,0)) \to \mathbb{R}$ via
\begin{equation}
 \mathfrak{D}_\beta[\zeta] = 
\begin{cases}
  \int_{-d}^0 \kappa\abs{\zeta'}^2 +  \beta_+ \abs{\zeta(0)}^2   + \beta_- \abs{\zeta(-d)}^2 &\text{if } \beta \in [0,\infty)^2 \\
  \int_{-d}^0 \kappa\abs{\zeta'}^2   + \beta_- \abs{\zeta(-d)}^2 &\text{if } \beta_+ = \infty, \beta_- \in [0,\infty) \\
   \int_{-d}^0 \kappa\abs{\zeta'}^2 +  \beta_+ \abs{\zeta(0)}^2    &\text{if } \beta_- = \infty, \beta_+ \in [0,\infty)  \\
   \int_{-d}^0 \kappa \abs{\zeta'}^2   &\text{if } \beta_+ = \beta_-= \infty.
\end{cases}
\end{equation}

We now prove a variational principle.

\begin{lemma}\label{eigenvalue}
Let 
\begin{equation}
 \mathcal{C} = \{\zeta \in  H^1_\beta((-d,0)) \;\vert\; \norm{\zeta}_{L^2((-d,0))} =1\}.
\end{equation}
Then there exists $\zeta_\beta \in \mathcal{C}$ such that 
\begin{equation}
 \mathfrak{D}_\beta[\zeta_\beta] = \min_{\mathcal{C}} \mathfrak{D}_\beta :=\mu_\beta.
\end{equation}
  Moreover, 
\begin{equation}
 \mu_\beta \ge 
\begin{cases}
\min\left\{\frac{\kappa \pi^2}{4d^2}, \frac{\beta_+ \beta_-}{\kappa} \right\}  &\text{if } \beta \in (0,\infty)^2 \\
 \left(\frac{4d^2}{\kappa \pi^2} + \frac{d}{\beta_-}\right)^{-1} &\text{if } \beta_+ =0, \beta_- \in (0,\infty) \\
  \left(\frac{4d^2}{\kappa \pi^2} + \frac{d}{\beta_+}\right)^{-1} &\text{if } \beta_- =0, \beta_+ \in (0,\infty) \\ 
 \frac{\kappa \pi^2}{4d^2} & \text{if } \beta_+ =\infty, \beta_- \in [0,\infty) \\
 \frac{\kappa \pi^2}{4d^2} & \text{if } \beta_- =\infty, \beta_+ \in [0,\infty) \\
  \frac{\kappa \pi^2} {d^2} &\text{if } \beta_+ = \beta_- =\infty  \\
 0 & \text{if } \beta_+ = \beta_- =0,
\end{cases}
\end{equation}
and in fact this inequality is an equality when  $\beta \in \{0,\infty\}^2$.
\end{lemma}
\begin{proof}

The existence of $\zeta_\beta$, a minimizer of $\mathfrak{D}_\beta$ over the constraint set $\mathcal{C}$, follows from a standard application of the direct methods in the calculus of variations, so we shall omit the details.  It's clear that $\mu_\beta := \mathfrak{D}_\beta[\zeta_\beta] \ge 0$.  The minimizer must then satisfy the equations
\begin{equation}\label{eigen_1}
 \begin{cases}
  -\kappa \zeta_\beta''(r) =  \mu \zeta_\beta(r) &\text{for } r\in (-d,0) \\
  \kappa \zeta_\beta'(0) + \beta_+ \zeta_\beta(0) =0 \\
  -\kappa \zeta_\beta'(-d) + \beta_- \zeta_\beta(-d) =0
 \end{cases}
\end{equation}
with the obvious reinterpretation when $\beta_\pm =\infty$.

When $\beta \in (0,\infty)^2$, the problem \eqref{eigen_1} corresponds to finding eigenfunctions of the Laplacian with Robin boundary conditions in one dimension.  Standard analysis (see for example, Chapter 4.3 of Strauss's book \cite{strauss}) shows that $\mu_\beta$ must be the smallest positive solution to 
\begin{equation}
 \tan\left(d \sqrt{\frac{\mu_\beta}{\kappa}} \right) = \frac{\sqrt{\mu_\beta \kappa}(\beta_+ + \beta_-)}{\kappa \mu_\beta - \beta_+ \beta_-}.
\end{equation}
Although this equation cannot be solved analytically in general, it is a simple matter to see that 
\begin{equation}
 \min\left\{\frac{\kappa \pi^2}{4d^2}, \frac{\beta_+ \beta_-}{\kappa} \right\}  \le \mu_\beta.
\end{equation}

When $\beta_+ =0$ and $\beta_- \in (0,\infty)$, the problem \eqref{eigen_1} requires that $\mu_\beta$ be the smallest positive solution to 
\begin{equation}
 \tan\left(d \sqrt{\frac{\mu_\beta}{\kappa}} \right) = \frac{\beta_-}{\sqrt{\mu_\beta \kappa}},
\end{equation}
which must lie in the interval $(0,\pi/2)$.  By employing the estimate 
\begin{equation}
 x \tan(x) \le \frac{x^2}{1-\left(\frac{2x}{\pi} \right)^2} \text{ for }x \in (0,\pi/2),
\end{equation}
we may show that $\mu_\beta$ satisfies the bound
\begin{equation}
\frac{1}{ \frac{4d^2}{\kappa \pi^2} + \frac{d}{\beta_-} } \le \mu_\beta.
\end{equation}
In the case $\beta_-=0$ and $\beta_+ \in (0,\infty)$ a similar analysis allows us to deduce that 
\begin{equation}
\frac{1}{ \frac{4d^2}{\kappa \pi^2} + \frac{d}{\beta_+} } \le \mu_\beta.
\end{equation}

The case $\beta_+ = \infty$ and $\beta_- \in [0,\infty)$  for problem \eqref{eigen_1} leads to the equation 
\begin{equation}
 \tan\left(d \sqrt{\frac{\mu_\beta}{\kappa}} \right) = -\frac{\sqrt{\mu_\beta \kappa}}{\beta_-}.
\end{equation}
Similarly, the case $\beta_- = \infty$ and $\beta_+ \in [0,\infty)$ leads to the equation 
\begin{equation}
 \tan\left(d \sqrt{\frac{\mu_\beta}{\kappa}} \right) = -\frac{\sqrt{\mu_\beta \kappa}}{\beta_+}.
\end{equation}
In either case we find that 
\begin{equation}
 \frac{\kappa \pi^2} {4d^2} \le \mu_\beta,
\end{equation}
with equality achieved if the finite $\beta$ term actually vanishes.  

The case $\beta_+ = \beta_- = \infty$  in \eqref{eigen_1} leads to the equation 
\begin{equation}\label{eigen_2}
 \sin\left(d \sqrt{\frac{\mu_\beta}{\kappa}} \right)=0.
\end{equation}
The smallest positive solution is then 
\begin{equation}
 \frac{\kappa \pi^2} {d^2} = \mu_\beta.
\end{equation}
The remaining cases  $\beta_+ = \beta_- = 0$ in \eqref{eigen_1} leads to the equation \eqref{eigen_2} as well, but the solution must only be non-negative, and hence $\mu_\beta =0$.

\end{proof}

As a consequence of the minimality of $\mu_\beta$ we deduce a corresponding one-dimensional coercivity estimate.

\begin{cor}\label{1d_coercive}
 Let $\mu_\beta \ge 0$ be as given in Lemma \ref{eigenvalue}.    If $\zeta \in  H^1_\beta((-d,0)) $, then 
\begin{equation}
 \mu_\beta \int_{-d}^0 \abs{\zeta}^2 \le \mathfrak{D}_\beta[\zeta].
\end{equation}
\end{cor}
\begin{proof}
The square of the $L^2$ norm and $\mathfrak{D}_\beta$ have the same homogeneity, and so Lemma \ref{eigenvalue} allows us to write
\begin{equation}
 \mu_\beta = \min_{\zeta \in \mathcal{C}} \mathfrak{D}_\beta[\zeta]  = \min_{\zeta \in H^1_\beta((-d,0))} \frac{\mathfrak{D}_\beta[\zeta]}{\norm{\zeta}_{L^2}^2}.
\end{equation}
Hence, for any $\zeta \in  H^1_\beta((-d,0)) $ we may estimate 
\begin{equation}
 \mu_\beta \le \frac{\mathfrak{D}_\beta[\zeta]}{\norm{\zeta}_{L^2}^2},
\end{equation}
which is the desired inequality.

\end{proof}

%======================
\subsection{Rigid boundary}
%======================

With the one-dimensional coercivity in hand, we can now derive the general coercivity estimate in the case of a rigid boundary.  We recall the definition of $H^1_\beta(\Omega)$ from \eqref{sobolev_beta} and $\mathcal{D}_\beta$ from \eqref{Dthe}.  We also define the vertical part of $\mathcal{D}_\beta$ according to
\begin{equation}\label{D_vert}
 \mathcal{V}_\beta[\varphi] = 
\begin{cases}
  \int_\Omega \kappa\abs{\p_3 \varphi}^2  + \beta_+ \int_{\Sigma_+} \abs{\varphi}^2 + \beta_- \int_{\Sigma_-} \abs{\varphi}^2 &\text{if } \beta \in [0,\infty)^2 \\
  \int_\Omega \kappa\abs{\p_3 \varphi}^2 + \beta_- \int_{\Sigma_-} \abs{\varphi}^2 &\text{if } \beta_+ = \infty, \beta_- \in [0,\infty) \\
   \int_\Omega \kappa\abs{\p_3 \varphi}^2 + \beta_+ \int_{\Sigma_+} \abs{\varphi}^2    &\text{if } \beta_- = \infty, \beta_+ \in [0,\infty)  \\
   \int_\Omega \kappa\abs{\p_3 \varphi}^2     &\text{if } \beta_+ = \beta_-= \infty.
\end{cases}
\end{equation}

\begin{Proposition}\label{rigid_coercive}
 Let $\mu_\beta \ge 0$ be as given in Lemma \ref{eigenvalue} and $\mathcal{V}_\beta$ be as defined in \eqref{D_vert}.   If $\varphi \in H^1_\beta(\Omega)$, then 
\begin{equation}\label{rc_00}
 \mu_\beta \int_\Omega \abs{\varphi}^2 \le \mathcal{V}_\beta[\varphi] \le \mathcal{D}_\beta[\varphi].
\end{equation}
Moreover, the inequality is sharp in the sense that 
\begin{equation}\label{rc_01}
 \mu_\beta = \inf\{ \mathcal{D}_\beta[\varphi] \;\vert\; \varphi \in H^1_\beta(\Omega) \text{ and } \norm{\varphi}_{L^2(\Omega)}=1\}.
\end{equation}
When $\Gamma = (L_1 \mathbb{T}) \times (L_2 \mathbb{T})$ this infimum is actually a minimum.

\end{Proposition}

\begin{proof}
The inequality $\mathcal{V}_\beta[\varphi] \le \mathcal{D}_\beta[\varphi]$ is trivial, so it suffices to prove only the first inequality in \eqref{rc_00}.   By Fubini's theorem, we may write 
\begin{equation}
 \mathcal{V}_\beta[\varphi] = \int_{\Gamma} \mathfrak{D}_\beta[\varphi(x',\cdot)] dx'.
\end{equation}
For a.e. $x' \in \Gamma$ we know that $\varphi(x',\cdot) \in H^1_\beta((-d,0))$, and for such $x'$ we may use Corollary \ref{1d_coercive} to estimate 
\begin{equation}
 \mu_\beta \int_{-d}^0 \abs{\varphi(x',x_3)}^2 dx_3 \le  \mathfrak{D}_\beta[\varphi(x',\cdot)].
\end{equation}
Hence 
\begin{equation}
 \mu_\beta \int_\Omega \abs{\varphi}^2  = \mu_\beta \int_\Gamma \int_{-d}^0 \abs{\varphi(x',x_3)}^2 dx_3 dx' 
\le \int_\Gamma  \mathfrak{D}_\beta[\varphi(x',\cdot)]  dx'=  \mathcal{V}_\beta[\varphi], 
\end{equation}
which is \eqref{rc_00}.

It remains to prove \eqref{rc_01}.  We must break to cases depending on whether $\Gamma = (L_1 \mathbb{T}) \times (L_2 \mathbb{T})$ or $\Gamma = \mathbb{R}^2$.  Assume initially that $\Gamma = (L_1 \mathbb{T}) \times (L_2 \mathbb{T})$.  Set
\begin{equation}
 \varphi(x) = \frac{1}{\sqrt{L_1 L_2}} \zeta_\beta(x_3),
\end{equation}
where $\zeta_\beta$ is as constructed in Lemma \ref{eigenvalue}.  Then 
\begin{equation}
 \int_\Omega \abs{\varphi(x)}^2 dx = \left(\int_{\Gamma} \frac{1}{L_1 L_2} dx' \right) \left( \int_{-d}^0 \abs{\zeta_\beta(x_3)}^2 dx_3\right) = 1.
\end{equation}
Also, 
\begin{equation}
 \mathcal{D}_\beta[\varphi] = \mathcal{V}_\beta[\varphi] = \int_{\Gamma} \frac{1}{L_1 L_2}  \mathfrak{D}_\beta[\zeta_\beta] dx' =  \mathfrak{D}_\beta[\zeta_\beta] = \mu_\beta,
\end{equation}
which proves that \eqref{rc_01} holds and is a minimum when $\Gamma$ is periodic.

Now consider the case $\Gamma = \mathbb{R}^2$.  Choose $\psi \in C_c^\infty(\mathbb{R}^2)$ such that $\norm{\psi}_{L^2(\mathbb{R}^2)}=1$.  Let $\zeta_\beta$ again come from Lemma \ref{eigenvalue} and for $\alpha \in \R_+$ set
\begin{equation}
 \varphi_\alpha(x) = \alpha \psi(\alpha x') \zeta_\beta(x_3).  
\end{equation}
Then we compute 
\begin{equation}
 \int_\Omega \abs{\varphi_\alpha(x)}^2 dx = \left(\int_{\R^2} \alpha^2 \abs{\psi(\alpha x')}^2 dx' \right) \left( \int_{-d}^0 \abs{\zeta_\beta(x_3)}^2 dx_3\right) = 1
\end{equation}
and 
\begin{equation}
 \mathcal{D}_\beta[\varphi_\alpha] = \mathfrak{D}_\beta[\zeta_\beta] + \alpha^2 \int_{\R^2} \abs{\nab_\ast \varphi(x')}^2 dx' 
=  \mu_\beta  + \alpha^2 \int_{\R^2} \abs{\nab_\ast \varphi(x')}^2 dx'
\end{equation}
where $\nab_\ast$ denotes the horizontal gradient.  Then 
\begin{equation}
 \lim_{\alpha \to 0 }\mathcal{D}_\beta[\varphi_\alpha] = \mu_\beta,
\end{equation}
which proves \eqref{rc_01} when $\Gamma = \R^2$.

\end{proof}

We know from Lemma \ref{eigenvalue} that $\mu_\beta >0$ when $\beta \neq (0,0)$.  In this case the conclusion of the proposition is non-trivial.  When $\beta=(0,0)$ the result is actually trivial since $\mu_0 =0$.  Our next goal is to derive a non-trivial estimate when $\beta = (0,0)$, under a stronger assumption on the functions in question.  We will only be able to do so in the case of a periodic horizontal cross section.

\begin{Proposition}\label{rigid_coercive_zero}
Assume that $\Gamma = (L_1 \mathbb{T}) \times (L_2 \mathbb{T})$ and $\beta =(0,0)$.   If $\varphi \in  H_\beta^1(\Omega)$ and $\int_\Omega \varphi =0$, then
\begin{equation}
 \bar{\mu}_0 \int_\Omega \abs{\varphi}^2 \le  \mathcal{D}_\beta[\varphi],
\end{equation}
where
\begin{equation}
  \bar{\mu}_0 = \kappa \pi^2 \min\left\{\frac{1}{d^2}, \frac{4}{L_1^2},\frac{4}{L_2^2} \right\}.
\end{equation}
Moreover, the inequality is sharp in the sense that 
\begin{equation} 
 \bar{\mu}_0 = \inf\{ \mathcal{D}_\beta[\varphi] \;\vert\; \varphi \in H^1_\beta(\Omega),  \norm{\varphi}_{L^2(\Omega)}=1, \text{and } \int_\Omega \varphi =0\}.
\end{equation}

\end{Proposition}
\begin{proof}
We define the constraint set 
\begin{equation}
 \mathcal{C} = \{ \psi \in H^1(\Omega) \;\vert\; \norm{\psi}_{L^2(\Omega)}=1 \text{ and } \int_\Omega\psi =0\}. 
\end{equation}
The direct method in the calculus of variations allows us to produce a minimizer of $\mathcal{D}_\beta$ over $\mathcal{C}$.  That is, we can find $\psi_0 \in \mathcal{C}$ such that 
\begin{equation}
 \mathcal{D}_\beta[\psi_0] = \min_{\psi \in \mathcal{C}} \mathcal{D}_\beta[\psi] =:\lambda.
\end{equation}
Standard arguments reveal that $\psi_0$ must satisfy
\begin{equation}
 \begin{cases}
-\kappa \Delta \psi_0 = \lambda \psi_0 &\text{in }\Omega \\
 \p_3 \psi_0 = 0 &\text{on } \Sigma_\pm.
 \end{cases}
\end{equation}
The spectrum of this problem can be computed explicitly by using a separation of variables:
\begin{equation}
 \lambda(n_1,n_2,m) = \kappa \left[ \frac{\pi^2 m^2}{d^2} + 4 \pi^2 \left( \frac{n_1^2}{L_1^2} + \frac{n_2^2}{L_2^2}\right)  \right]
\end{equation}
for $n_1,n_2,m \in \mathbb{N}$. However, the condition $\int_\Omega \psi_0$ eliminates the zero eigenvalue as a possibility, and so $\lambda$ must be the second smallest eigenvalue, which is easily computed: 
\begin{equation}
 \lambda = \kappa \pi^2 \min\left\{\frac{1}{d^2}, \frac{4}{L_1^2},\frac{4}{L_2^2} \right\}.
\end{equation}

\end{proof}

We might try something similar when $\Gamma = \R^2$.  The problems then are two-fold.  First we don't know for sure that $\varphi \in L^1(\Omega)$, and so the condition $\int_\Omega \varphi =0$ need not make sense.  Second, and more fundamental, is that when $\beta=(0,0)$
\begin{equation}
 \inf\{ \mathcal{D}_\beta[\varphi] \;\vert\; \varphi \in C_c^\infty(\R^2 \times [-d,0]),  \norm{\varphi}_{L^2(\Omega)}=1, \text{and } \int_\Omega \varphi =0\}=0.
\end{equation}
Indeed we may use $\varphi_\alpha$ from  the proof of Proposition \ref{rigid_coercive} with the extra assumption that $\int_{\R^2} \psi =0$.  Then $\varphi_\alpha \in C_c^\infty(\R^2 \times [-d,0])$ and
\begin{equation}
 \int_\Omega \abs{\varphi_\alpha}^2 =1, \int_\Omega \varphi_\alpha =0, \text{ and } \lim_{\alpha \to 0} \mathcal{D}_\beta[\varphi_\alpha] = \mu_\beta =0.
\end{equation}
The upshot of this analysis is that the dissipation functional $\mathcal{D}_\beta$ is simply not $L^2$-coercive when $\Gamma = \R^2$ and $\beta =(0,0)$.

%======================
\subsection{Moving boundary }
%======================

Now we derive a coercivity estimate in the case of a moving boundary.

\begin{Proposition}\label{geo_coercive}
Let $J$ and $\a$ be as in \eqref{A_def} and \eqref{ABJ_def} and suppose that 
\begin{equation}
1/c_0\leq J\leq c_0 \text{ for some } c_0>1. 
\end{equation}
Let $\mathcal{M}^t_\beta$ be given by \eqref{Dthe2}.  If $\varphi \in H^1_\beta(\Omega)$, then  
\begin{equation}\label{core2}
 \frac{\mu_\beta}{c_0^2} \int_{\Omega} \abs{\varphi}^2 J dx \leq  \mathcal{M}^t_\beta[\varphi].
\end{equation}
\end{Proposition}

\begin{proof} 
Lemma \ref{geometric_ids} tells us that $\naba \varphi \cdot e_3 = K \p_3 \varphi$, and so the equality $K= J^{-1}$ allows us to estimate
\begin{equation}
 \int_\Omega J \abs{\naba \varphi}^2 \ge \int_\Omega J \abs{K \p_3 \varphi}^2 = \int_\Omega K \abs{\p_3 \varphi}^2 \ge \frac{1}{c_0} \int_\Omega \abs{\p_3 \varphi}^2.
\end{equation}
Similarly, 
\begin{equation}
 \int_{\Sigma_-} \beta_- \abs{\varphi}^2 K \ge \frac{1}{c_0}  \int_{\Sigma_-} \beta_- \abs{\varphi}^2.
\end{equation}
Since $c_0 >1$ and  $\abs{\n} = \sqrt{1 + \abs{\p_1 \eta}^2 + \abs{\p_2 \eta}^2} \ge 1$ on $\Sigma_+$, we may estimate 
\begin{equation}
\int_{\Sigma_+} \beta_+ \abs{\varphi}^2 \abs{\n} \ge \frac{1}{c_0} \int_{\Sigma_+} \beta_+ \abs{\varphi}^2.
\end{equation}
Combining these three inequalities yields the estimate 
\begin{equation}\label{gc_1}
\frac{1}{c_0} \mathcal{V}_\beta[\varphi]  \le  \mathcal{M}^t_\beta[\varphi],
\end{equation}
where $\mathcal{V}_\beta$ is defined by \eqref{D_vert}.  The desired estimate then follows \eqref{gc_1} and from Proposition \ref{rigid_coercive}.
  
\end{proof}

Again this estimate is only useful when $\beta \neq (0,0)$.  When $\beta =(0,0)$ we must use Proposition \ref{rigid_coercive_zero} as the basis of our estimate.

\begin{Proposition}\label{geo_coercive_zero}
Assume that $\beta =(0,0)$ and that $\Gamma = (L_1 \mathbb{T}) \times (L_2 \mathbb{T})$.  Let $J$ and $\a$ be as in \eqref{A_def} and \eqref{ABJ_def} and suppose that 
\begin{equation}\label{geo_c_00}
1/c_0\leq J\leq c_0  \text{ for some } c_0>1 
\end{equation}
and
\begin{equation}\label{geo_c_01}
\frac{1}{c_1} I \le  J \a^T \a \le c_1 I \text{ for some } c_1 >0.
\end{equation}
Let $\mathcal{M}^t_\beta$ be given by \eqref{Dthe2}.  If $\varphi \in H^1_\beta(\Omega)$ satisfies $\int_\Omega J \psi =0$, then  
\begin{equation}\label{geo_c_02}
 \frac{\bar{\mu}_0}{c_0 c_1} \int_{\Omega} \abs{\varphi}^2 J dx \leq  \mathcal{M}^t_\beta[\varphi],
\end{equation}
where $\bar{\mu}_0$ is as defined in Proposition \ref{rigid_coercive_zero}.
\end{Proposition}
\begin{proof}

Due to the inequalities \eqref{geo_c_00} and \eqref{geo_c_01}, we may endow the space $H^1(\Omega)$ with the inner-product 
\begin{equation}
 \br{\psi,\varphi}_{1,t} = \int_\Omega J(t,\cdot)\psi \varphi  + J(t,\cdot) \nab_{\a(t,\cdot)} \psi \cdot \nab_{\a(t,\cdot)} \varphi, 
\end{equation}
and the resulting norm is equivalent to the standard $H^1$ norm. Similarly, $L^2(\Omega)$ may be endowed with the inner-product
\begin{equation}
 \br{\psi,\varphi}_{0,t} = \int_\Omega J(t,\cdot)\psi \varphi.
\end{equation}
It's clear that 
\begin{equation}
 \min\left\{ \left. \frac{\mathcal{M}_\beta^t[\psi]}{\ns{\psi}_{0,t}} \;\right\vert\; \psi \in H^1(\Omega) \backslash\{0\}\right\} =0
\end{equation}
and that the minimizer is constant function $\varphi_0$.  We may use the direct methods in the calculus of variations to find $\varphi_1 \in H^1(\Omega)\backslash \{0\}$ such that
\begin{equation}
\frac{\mathcal{M}_\beta^t[\psi_1]}{\ns{\psi_1}_{0,t}} =
 \min\left\{ \left. \frac{\mathcal{M}_\beta^t[\psi]}{\ns{\psi}_{0,t}} \;\right\vert\; \psi \in H^1(\Omega) \backslash\{0\} \text{ and } \br{\psi,\varphi_0}_{0,t}=0 \right\} =:\lambda_1(t).
\end{equation}

For any $w \in H^1(\Omega) \backslash \{0\}$ we will  write 
\begin{equation}
 X_t(w) = \{ \psi \in H^1(\Omega) \;\vert\; \br{\psi,w}_{0,t}=0\}
\end{equation}
for the co-dimension one subspace that is perpendicular to $w$ with respect to the inner-product $\br{\cdot,\cdot}_{0,t}$.  We claim that $\psi_1$ satisfies the maximin principle 
\begin{equation}\label{geo_c_1}
\max_{w \in H^1(\Omega) \backslash \{0\}} \min_{ \psi \in X_t(w) \backslash\{0\} } \frac{\mathcal{M}_\beta^t[\psi]}{\ns{\psi}_{0,t}} = \frac{\mathcal{M}_\beta^t[\psi_1]}{\ns{\psi_1}_{0,t}}.
\end{equation}
To prove the claim we first set $V = \text{span}\{\varphi_0,\varphi_1\}$.  The definitions of $\varphi_0,\varphi_1$ easily imply that
\begin{equation}
 \max_{\psi \in V \backslash \{0\}} \frac{\mathcal{M}_\beta^t[\psi]}{\ns{\psi}_{0,t}} = \lambda_1(t).
\end{equation}
Let $w \in H^1(\Omega) \backslash \{0\}$ be arbitrary.   Since $V$ is of dimension $2$ it must hold that $V \cap X_t(w) \neq \{0\}$.  Choosing $v \in V \cap X_t(w) \backslash \{0\}$, we then find that 
\begin{equation}
 \frac{\mathcal{M}_\beta^t[v]}{\ns{v}_{0,t}} \le  \max_{\psi \in V \backslash \{0\}} \frac{\mathcal{M}_\beta^t[\psi]}{\ns{\psi}_{0,t}} = \lambda_1(t),
\end{equation}
and hence that
\begin{equation}
 \min_{ \psi \in X_t(w) \backslash\{0\} } \frac{\mathcal{M}_\beta^t[\psi]}{\ns{\psi}_{0,t}} \le \lambda_1(t).
\end{equation}
From this and the definition of $\lambda_1(t)$ as a minimizer over $X_t(\varphi_0)$, we deduce that \eqref{geo_c_1} holds, proving the claim.

Now we return to the proof of Proposition \ref{rigid_coercive_zero}; we may use the matching homogeneities of $\mathcal{D}_\beta[\cdot]$ and $\norm{\cdot}_{L^2(\Omega)}^2$ to rewrite
\begin{equation}
 \bar{\mu}_0 =  \min\left\{ \left. \frac{\mathcal{D}_\beta[\psi]}{\ns{\psi}_{L^2(\Omega)}} \;\right\vert\; \psi \in H^1(\Omega) \backslash\{0\} \text{ and } \int_\Omega \psi =0 \right\}.
\end{equation}
Since $JK=1$ we have that 
\begin{equation}
 \int_\Omega \psi =0 \Leftrightarrow \int_\Omega \psi K J=0 \Leftrightarrow \br{\psi, K}_{0,t}=0  \Leftrightarrow \psi \in X_t(K).
\end{equation}
Hence
\begin{equation}\label{geo_c_2}
  \bar{\mu}_0  = \min_{\psi \in X_t(K) \backslash \{0\}} \frac{\mathcal{D}_\beta[\psi]}{\ns{\psi}_{L^2(\Omega)}}.
\end{equation}

The inequalities  \eqref{geo_c_00} and \eqref{geo_c_01} allow us to estimate
\begin{equation}
\frac{1}{c_1} \int_\Omega \abs{\nab \psi}^2 \le \int_\Omega J \a^T \a \nab \psi \cdot \nab\psi =   \int_\Omega J \abs{\naba \psi}^2 = \mathcal{M}_\beta^t[\psi] 
\end{equation}
and 
\begin{equation}
 \frac{1}{c_0}    \int_\Omega J \abs{\psi}^2 \le \int_\Omega \abs{\psi}^2 
\end{equation}
for every $\psi \in H^1(\Omega)$.  Consequently, 
\begin{equation}\label{geo_c_3}
\frac{1}{c_0 c_1} \frac{\mathcal{D}_\beta[\psi]}{\ns{\psi}_{L^2(\Omega)} }\le  \frac{\mathcal{M}_\beta^t[\psi]}{\ns{\psi}_{0,t}} 
\end{equation}
for $\psi \in H^1(\Omega) \backslash \{0\}.$  Combining \eqref{geo_c_1}, \eqref{geo_c_2}, and \eqref{geo_c_3} then yields the bound
\begin{equation}
\begin{split}
 \frac{\bar{\mu}_0}{c_0 c_1}   &= \min_{\psi \in X_t(K) \backslash \{0\}} \frac{1}{c_0c_1}\frac{\mathcal{D}_\beta[\psi]}{\ns{\psi}_{L^2(\Omega)}} \\
&\le \min_{\psi \in X_t(K) \backslash \{0\}} \frac{\mathcal{M}_\beta^t[\psi]}{\ns{\psi}_{0,t}} \\
&\le \max_{w \in H^1(\Omega) \backslash \{0\}} \min_{ \psi \in X_t(w) \backslash\{0\} } \frac{\mathcal{M}_\beta^t[\psi]}{\ns{\psi}_{0,t}} \\
& = \frac{\mathcal{M}_\beta^t[\psi_1]}{\ns{\psi_1}_{0,t}} \\
& =  \min\left\{ \left. \frac{\mathcal{M}_\beta^t[\psi]}{\ns{\psi}_{0,t}} \;\right\vert\; \psi \in H^1(\Omega) \backslash\{0\} \text{ and } \br{\psi,\varphi_0}_{0,t}=0 \right\}.
 \end{split}
\end{equation}
Since $\psi_0$ is constant, this implies \eqref{geo_c_02} when $\psi \neq 0$.  The case $\psi=0$ is trivial.

\end{proof}

%======================
%======================
\section{Proof of Main results}
%======================
%======================

We start with the proof of Theorem \ref{thm-ri}. 

\begin{proof}[Proof of Theorem \ref{thm-ri}]
We will only prove the  estimates; with them in hand the existence of global solutions then follows from Proposition \ref{rigid_lwp} and a standard continuation argument.  The global solutions must then satisfy the decay estimates as well.  The argument for $(iii)$ is somewhat simpler, so we begin with it.

From Lemma \ref{lem1} and Proposition \ref{rigid_coercive}, we first obtain 
\begin{equation}\label{tri_1}
\begin{split}
\frac12\frac{d}{dt} \int_{\Omega} |\theta-\theta^\beta_{eq}|^2 dx  +\mu_\beta \int_{\Omega} |\theta-\theta^\beta_{eq}|^2 dx & \leq -  \p_3 \theta_{eq}^\beta  \int_{\Omega} (\theta-\theta^\beta_{eq})\,  u_3 dx \\
& \leq \abs{\p_3 \theta_{eq}^\beta}   \left( \int_{\Omega} |\theta-\theta^\beta_{eq}|^2 dx\right)^\frac12
 \left( \int_\Omega |u|^2 dx \right)^\frac12. 
\end{split}
\end{equation}
Let $f(t):= \| \theta(t,\cdot)-\theta^\beta_{eq} \|_{L^2(\Omega)}$.  Since $\norm{u(t)}_{L^2(\Omega)} \le  g(t)$,   we deduce from \eqref{tri_1} that $f$ satisfies
\begin{equation}
\frac{df}{dt} + \mu_\beta f \leq \abs{\p_3 \theta_{eq}^\beta}  g(t)
\end{equation}
which leads to 
\begin{equation}
\frac{d}{dt} \left( e^{ \mu_\beta t} f \right) \leq  \abs{\p_3 \theta_{eq}^\beta}  e^{ \mu_\beta t} g(t). 
\end{equation}
By integrating in $t$, we obtain \eqref{decay1}.

Next we prove $(i)$, in which case $\beta =(0,0)$ and $\Gamma = (L_1 \mathbb{T}) \times (L_2 \mathbb{T})$.  We use Lemma \ref{lem1_zero} with $C = \frac{1}{\abs{\Omega}} \int_\Omega \theta_0$ to find that 
\begin{equation}\label{tri_2}
\hal\frac{d}{dt} \int_{\Omega} |\theta-\theta^\beta_{eq} -C |^2 dx  + \mathcal{D}_\beta[\theta - \theta_{eq}^\beta -C] = 0.
\end{equation}
We then know from \eqref{l1z_02} of Lemma \ref{lem1_zero}   that 
\begin{equation}
 \int_\Omega (\theta(t,\cdot)-\theta^\beta_{eq} -C)dx =0
\end{equation}
for all $t$.  Proposition \ref{rigid_coercive_zero} then allows us to deduce from \eqref{tri_2} that
\begin{equation}
 \hal\frac{d}{dt} \int_{\Omega} |\theta-\theta^\beta_{eq} -C |^2 dx  + \bar{\mu}_0 \int_{\Omega} |\theta-\theta^\beta_{eq} -C |^2 dx \le 0.
\end{equation}
The estimate \eqref{decay00} then follows from this inequality as above.

Finally, to prove $(ii)$ we simply integrate the identity \eqref{ene1_n}.  

\end{proof}

We next present the proof of Theorem \ref{thm-mo}. 

\begin{proof}[Proof of Theorem \ref{thm-mo}]

As in the proof of Theorem \ref{thm-mo}, we will only prove the  estimates.  The global existence claim follows from these and a continuation argument.  

We start with $(i)$. When $\beta=(0,0)$, we have the energy identity \eqref{gez_01} and the zero-average condition \eqref{gez_03} from Lemma \ref{geo_energy_zero}.  Together with the coercivity estimate \eqref{geo_c_01} of Proposition \ref{geo_coercive_zero}, this implies that 
\begin{equation}
\frac12\frac{d}{dt} \int_{\Omega} |\theta- \theta_{eq}^\beta -\theta_{avg}|^2 J dx  + \frac{\bar{\mu}_0}{ c_0c_1} \int_{\Omega}  |\theta- \theta_{eq}^\beta - -\theta_{avg} |^2 J dx  \leq 0,
\end{equation}
which easily implies \eqref{decay0}. 

To prove $(ii)$ we use the identity \eqref{ene2_n}, which followed from Lemma \ref{geo_energy}, and simply integrate in time.

For $(iii)$, we first note that from \eqref{ene2}, \eqref{core2}, and the Cauchy-Schwarz inequality,  
\begin{equation}
\begin{split}
\frac12\frac{d}{dt}&\int_\Omega  | \theta  -\theta_{eq}^\infty|^2J dx +\frac{\mu_\beta}{c_0^2 } \int_{\Omega}  |\theta -\theta_{eq}^\infty  |^2 J dx\\
&\leq  \p_3 \theta_{eq}^\beta 
 \int_\Omega     (\theta -\theta_{eq}^\infty) \left\{   \dt \bar{\eta} K \tilde{d}  -u_j  \a_{j3}  +\kappa  \a_{jl} \p_l\a_{j3}  \right\}  J dx \\
& \le \abs{\p_3 \theta_{eq}^\beta} \left( \int_\Omega  |\theta -\theta_{eq}^\infty|^2 J dx\right)^{\frac{1}{2}} \left( \int_\Omega   \abs{\dt \bar{\eta} K\tilde{d}  -u_j  \a_{j3}  +\kappa  \a_{jl} \p_l\a_{j3}}^2 J dx\right)^{\frac{1}{2}}
 \end{split}
\end{equation}
Lemmas \ref{eta_poisson} and \ref{eta_small} allow us to estimate 
\begin{equation}
\begin{split}
 \norm{\left(\dt \bar{\eta} K\tilde{d}  -u_j  \a_{j3}  +\kappa  \a_{jl} \p_l\a_{j3} \right) \sqrt{J}}_{L^2(\Omega)} 
&\le C \left( \norm{\dt \bar{\eta}}_{L^2(\Omega)} + \norm{u}_{L^2(\Omega)} + \norm{\nab \a}_{L^2(\Omega)} \right) \\
&\le C \left( \norm{\dt \eta}_{L^2(\Gamma)} + \norm{u}_{L^2(\Omega)} + \norm{\nab_\ast \eta}_{H^{1/2}(\Gamma)} \right) \\
&\le C g(t).
\end{split}
\end{equation}
Combining these,  we obtain 
\begin{equation}
\begin{split}
\frac12\frac{d}{dt}\int_\Omega  | \theta  -\theta_{eq}^\infty|^2J dx +\frac{\mu_\beta}{c_0^2 } \int_{\Omega}  |\theta -\theta_{eq}^\infty  |^2 J dx \leq C \abs{\p_3 \theta_{eq}^\beta}\left( \int_\Omega  |\theta -\theta_{eq}^\infty|^2 J dx\right)^\frac12  g(t),
\end{split}
\end{equation}
for a universal constant $C>0$.  Let $f:= \| (\theta-\theta^\infty_{eq})\sqrt{J} \|_{L^2(\Omega)}$. Then $f$ satisfies 
\begin{equation}
\frac{df}{dt} + \frac{\mu_\beta}{c_0^2 } f \leq C \abs{\p_3 \theta_{eq}^\beta} g(t) \quad\Longrightarrow \quad \frac{d}{dt}\left( e^{\frac{\mu_\beta}{c_0^2 }t } f\right)\leq C  \abs{\p_3 \theta_{eq}^\beta}  e^{\frac{\mu_\beta}{c_0^2 }t } g(t). 
\end{equation}
By integrating in $t$, we deduce \eqref{decay2}.

In order to establish $(iv)$, we first use the energy identity \eqref{ene2}  to obtain 
\begin{equation}\label{tmo_1} 
\begin{split}
\frac12\frac{d}{dt}&\int_\Omega  | \theta  -\theta_{eq}^\beta|^2J dx +  \mathcal{M}_\beta^t[\theta - \theta_{eq}^\beta]  \\
&\leq   \p_3 \theta_{eq}^\beta  \int_\Omega     (\theta -\theta_{eq}^\beta) \left\{   \dt \bar{\eta} \tilde{d}  -u_j J \a_{j3}  +\kappa J \a_{jl} \p_l\a_{j3}  \right\}  dx \\
 &\quad + \int_{\Sigma_+} \beta_+ (\bar{\theta} - \theta_{eq}^\beta) (\theta - \theta_{eq}^\beta) \abs{\n} (1- K \abs{\n}) \\
 &=: I + II . 
\end{split}
\end{equation}
For $I$  we argue as in case $(iii)$ to bound
\begin{equation}
 \abs{I} \le C \abs{\p_3 \theta_{eq}^\beta}\left( \int_\Omega  |\theta -\theta_{eq}^\infty|^2 J dx\right)^\frac12  g(t)
\end{equation}
for a universal constant $C>0$.  For $II$ we employ Cauchy's inequality, \eqref{non_es_01} of Lemma \ref{nonlin_ests}, and \eqref{es_01} of Lemma \ref{eta_small}   to bound
\begin{equation}\label{tmo_2}
\begin{split}
 \abs{II} & \le \hal \int_{\Sigma_+} \beta_+ \abs{\theta- \theta_{eq}^\beta}^2 \abs{\n} + \frac{\beta_+  }{2}\abs{\bar{\theta}-\theta_{eq}^\beta}^2 \int_{\Sigma_+} \abs{\n} \abs{1-K \abs{\n}}^2 \\
 &\le  \hal \mathcal{M}_\beta^t[\theta-\theta_{eq}^\beta] + C \frac{\beta_+  }{2}\abs{\bar{\theta}-\theta_{eq}^\beta}^2 \norm{\eta}_{H^1(\Gamma)}^2 \\
 &\le  \hal \mathcal{M}_\beta^t[\theta-\theta_{eq}^\beta] + C \frac{\beta_+  }{2}\abs{\bar{\theta}-\theta_{eq}^\beta}^2 h^2(t)
\end{split}
\end{equation}
for a universal constant $C>0$.  From \eqref{tmo_1}--\eqref{tmo_2}  we  deduce that 
\begin{multline}
\hal \frac{d}{dt}\int_\Omega  | \theta  -\theta_{eq}^\beta|^2J dx + \hal \mathcal{M}_\beta^t[\theta-\theta_{eq}^\beta] \\
\le C \abs{\p_3 \theta_{eq}^\beta}\left( \int_\Omega  |\theta -\theta_{eq}^\infty|^2 J dx\right)^\frac12  h(t) +  C \frac{\beta_+  }{2}\abs{\bar{\theta}-\theta_{eq}^\beta}^2 h^2(t).
\end{multline}
We then employ Proposition \ref{geo_coercive} and Cauchy's inequality to estimate 
\begin{multline}
\hal \frac{d}{dt}\int_\Omega  | \theta  -\theta_{eq}^\beta|^2J dx + \frac{\mu_\beta}{2c_0^2} \int_\Omega  | \theta  -\theta_{eq}^\beta|^2J  \le \frac{\mu_\beta}{4 c_0^2 }\int_\Omega  | \theta  -\theta_{eq}^\beta|^2J   \\
 +C \left( \abs{\p_3 \theta_{eq}^\beta}^2  \frac{c_0^2}{\mu_\beta}  + \frac{\beta_+  }{2}\abs{\bar{\theta}-\theta_{eq}^\beta}^2 \right) h^2(t).
\end{multline}
Hence 
\begin{equation}\label{tmo_3}
  \frac{d}{dt}\int_\Omega  | \theta  -\theta_{eq}^\beta|^2J dx + \frac{\mu_\beta}{2c_0^2} \int_\Omega  | \theta  -\theta_{eq}^\beta|^2J  \le C \left( \abs{\p_3 \theta_{eq}^\beta}^2  \frac{c_0^2}{\mu_\beta}  + \frac{\beta_+  }{2}\abs{\bar{\theta}-\theta_{eq}^\beta}^2 \right) h^2(t)
\end{equation}
for some universal $C>0$.  Then \eqref{decay3} follows from \eqref{tmo_3} and Gronwall's inequality.

\end{proof}

%======================
%======================
%\section{Discussion}
%======================
%======================

\appendix

%======================
%======================
\section{Poisson extension}
%======================
%======================

In this section we define the Poisson extension operator when $\Gamma= \R^2$ and when $\Gamma =  (L_1 \mathbb{T}) \times (L_2 \mathbb{T})$, and we record some useful estimates.  

For a function $f: \R^2 \to \R$  the Poisson extension in $\Rn{2} \times (-\infty,0)$  is defined by
\begin{equation}\label{poisson_def_inf}
 \mathcal{P}f(x',x_3) = \int_{\Rn{2}} \hat{f}(\xi) e^{2\pi \abs{\xi}x_3} e^{2\pi i x' \cdot \xi} d\xi,
\end{equation}
where we have employed the Fourier transform
\begin{equation}
 \hat{f}(\xi) = \int_{\R^2} f(x') e^{-2\pi i x' \cdot \xi} dx'
\end{equation}
for $\xi \in \R^2$.  On the other hand, for a function $f: (L_1 \mathbb{T}) \times (L_2 \mathbb{T}) \to \R$ we define the Poisson extension in $\Omega = (L_1 \mathbb{T}) \times (L_2 \mathbb{T}) \times (-\infty,0)$ by
\begin{equation}\label{poisson_def_per}
\mathcal{P} f(x',x_3) = \sum_{n \in   (L_1^{-1} \mathbb{Z}) \times (L_2^{-1} \mathbb{Z}) }  e^{2\pi i n \cdot x'} e^{2\pi \abs{n}x_3} \hat{f}(n).
\end{equation}
Here, for $n \in   (L_1^{-1} \mathbb{Z}) \times (L_2^{-1} \mathbb{Z})$ we have written
\begin{equation}
 \hat{f}(n) = \int_\Gamma f(x')  \frac{e^{-2\pi i n \cdot x'}}{L_1 L_2} dx'
\end{equation}
for the Fourier coefficient.

Although $\mathcal{P} f$ is defined in all of $\Gamma \times (-\infty,0)$, we will only need bounds on its norm in the restricted domain $\Omega = \Gamma \times (-d,0)$.  This yields a couple improvements of the usual estimates of $\mathcal{P} f$ on the set $\Gamma \times (-\infty,0)$.

\begin{lemma}\label{poisson_ests}
Assume that $\Gamma = \R^2$ or else $\Gamma = (L_1 \mathbb{T}) \times (L_2 \mathbb{T})$.  Let $\mathcal{P} f$ be the Poisson integral of a function $f: \Gamma \to \R$ define by either \eqref{poisson_def_inf} or \eqref{poisson_def_per}.  Suppose that $f$ belongs to either $\dot{H}^{q}(\Gamma)$ or else $\dot{H}^{q-1/2}(\Gamma)$ for some $q \in \mathbb{N}$ (here $\dot{H}^s$ is the usual homogeneous Sobolev space of order $s$).  Then
\begin{equation}\label{pe_01} 
 \norm{\nab^q \mathcal{P}f }_{L^2(\Omega)} \ls \norm{f}_{\dot{H}^{q-1/2}(\Gamma)} \text{ and }  \norm{\nab^q \mathcal{P}f }_{L^2(\Omega)} \ls \norm{f}_{\dot{H}^{q}(\Gamma)}.
\end{equation}
Also, for $q\in \mathbb{N}$ we have the estimate 
\begin{equation}\label{pe_02}
 \norm{\nab^q \mathcal{P}f }_{L^\infty(\Omega)} \ls \norm{f}_{H^{q+3/2}(\Gamma)}.
\end{equation}
\end{lemma}
\begin{proof}
 These estimates are proved in \cite{GT_lwp}: Lemmas A.7 and A.8 handle the case $\Gamma = \R^2$ and Lemmas A.9 and A.10 handle the case $\Gamma = (L_1 \mathbb{T}) \times (L_2 \mathbb{T})$.
\end{proof}

%======================
%======================
\section{Some tools for the moving boundary problem}
%======================
%======================

In this section we record a number of results that are essential the analysis of the moving boundary problem.  We begin with some applications of Lemma \ref{poisson_ests} to $\eta$ and $\dt \eta$.

\begin{lemma}\label{eta_poisson}
We have the estimates 
\begin{equation}
\begin{split}
 &\norm{\bar{\eta}}_{H^3(\Omega)} \ls  \norm{\eta}_{H^{5/2}(\Gamma)} \\
 &\norm{\dt \bar{\eta} }_{L^2(\Omega) } \ls \norm{\dt \eta}_{L^2(\Gamma)} \\
 &\norm{\dt \bar{\eta}}_{L^\infty(\Omega) } \ls \norm{\dt \eta}_{H^{3/2}(\Gamma)}.
\end{split}
\end{equation}
\end{lemma}
\begin{proof}
 The first two estimates follow from \eqref{pe_01} of Lemma \ref{poisson_ests}, while the third follows from \eqref{pe_02}.  
\end{proof}

Our next result provides some useful estimates for the terms  $A,$ $B,$ $J,$ $K,$  $\mathcal{A}$, and $\n$ defined in \eqref{A_def}, \eqref{ABJ_def}, and \eqref{N_def}.

\begin{lemma}\label{eta_small}
There exists a universal $0 < \delta < 1$ so that if $\norm{\eta}_{H^{5/2}(\Gamma)} \le \delta$, then 
\begin{equation}\label{es_01}
\begin{split}
 & \norm{J-1}_{L^\infty(\Omega)} +\norm{A}_{L^\infty(\Omega)} + \norm{B}_{L^\infty(\Omega)} \le \hal, \\
 & \norm{\n-1}_{L^\infty(\Gamma)} + \norm{K-1}_{L^\infty(\Gamma)} \le \hal, \text{ and }  \\
 & \norm{K}_{L^\infty(\Omega)} + \norm{\mathcal{A}}_{L^\infty(\Omega)} \ls 1.
 \end{split}
\end{equation}
Moreover, 
\begin{equation}\label{es_02}
 \hal \int_\Omega \abs{\varphi}^2 \le \int_\Omega J \abs{ \varphi}^2 \le 2 \int_\Omega \abs{\varphi}^2,
\end{equation}
\begin{equation}
 \hal \int_\Omega \abs{\nab \varphi}^2 \le \int_\Omega J \abs{\naba \varphi}^2 \le 2 \int_\Omega \abs{\nab \varphi}^2,
\end{equation}
\begin{equation}
 \hal \int_{\Sigma_+} \abs{\varphi}^2 \le \int_{\Sigma_+}  \abs{\varphi}^2 \abs{\n} \le 2 \int_{\Sigma_+} \abs{\varphi}^2 
\end{equation}
\begin{equation}\label{es_03}
 \hal \int_{\Sigma_-} \abs{\varphi}^2 \le \int_{\Sigma_-}  \abs{\varphi}^2 K \le 2 \int_{\Sigma_-} \abs{\varphi}^2 
\end{equation}
for all $\varphi \in H^1(\Omega)$.
\end{lemma}
\begin{proof}
The estimate \eqref{es_01} is guaranteed by Lemma 2.6 of \cite{GT_inf}  when $\Gamma = \mathbb{R}^2$ and by Lemma 2.4 of \cite{GT_per} when $\Gamma = (L_1 \mathbb{T}) \times (L_2 \mathbb{T})$.  The estimates \eqref{es_02}--\eqref{es_03} then follow from \eqref{es_01} as in Lemma 2.1 of \cite{GT_lwp}.
\end{proof}

Next we record an estimate of $\mathcal{A}$ in $H^1(\Omega)$.

\begin{lemma}\label{nonlin_ests}
Let $\delta \in (0,1)$ be as in Lemma \ref{eta_small} and suppose that  $\norm{\eta}_{H^{5/2}(\Gamma)}^2 \le \delta$.  Then
\begin{equation}\label{non_es_00}
 \norm{\nab \mathcal{A} }_{L^2(\Omega)} \ls  \norm{\nab_\ast \eta}_{H^{1/2}(\Gamma)}  
\end{equation}
where $\nab_\ast \eta = (\p_1 \eta,\p_2 \eta)$ is the horizontal gradient; also 
\begin{equation}\label{non_es_01}
 \norm{1 - K \abs{\n}}_{L^2(\Gamma)} \ls \norm{\eta}_{H^{1}(\Gamma)}.
\end{equation}

\end{lemma}
\begin{proof}
According to \eqref{A_def} and \eqref{ABJ_def} we may trivially estimate 
\begin{equation}
 \begin{split}
 \abs{\nab \mathcal{A}}^2 &= \abs{\nab(AK)}^2 + \abs{\nab(BK)}^2 + \abs{\nab K}^2  \\
&\ls \left( \abs{\nab A}^2 + \abs{\nab B}^2 \right)\abs{K}^2 + (1 + \abs{A}^2 + \abs{B}^2) \abs{K}^4 \abs{\nab J}^2 \\
&\ls \left( \abs{\nab \bar{\eta} }^2 + \abs{\nab^2 \bar{\eta} }^2 \right) \left(\abs{K}^2 + \abs{K}^4(1 + \abs{A}^2 + \abs{B}^2) \right). 
 \end{split}
\end{equation}
We may then use the  $L^\infty(\Omega)$ estimates in \eqref{es_01} of Lemma \ref{eta_small} to estimate 
\begin{equation}\label{non_es_1}
 \norm{\nab \mathcal{A}}_{L^2(\Omega)}^2 \ls     \norm{\nab \bar{\eta} }_{L^2(\Omega)}^2 +  \norm{\nab^2 \bar{\eta} }_{L^2(\Omega)}^2.
\end{equation}
The estimate \eqref{non_es_00} then follows from \eqref{non_es_1} and \eqref{pe_01} of Lemma \ref{poisson_ests}.

Next, we rewrite
\begin{equation}
 1- K \abs{\n} = 1-K - K(\abs{\n}-1) = K(J-1) - K (\abs{\n}-1).
\end{equation}
Then Lemma \ref{eta_small} allows us to estimate 
\begin{equation}\label{non_es_2}
  \norm{1 - K \abs{\n}}_{L^2(\Gamma)}  \ls  \norm{J-1}_{L^2(\Gamma)}  +  \norm{1 -  \abs{\n}}_{L^2(\Gamma)}.
\end{equation}
We then use the definition of $J$ and the Poisson extension to estimate 
\begin{equation}
 \norm{J-1}_{L^2(\Gamma)}  \ls \norm{\eta}_{H^1(\Gamma)}. 
\end{equation}
On the other hand,
\begin{equation}
 \abs{1 -  \abs{\n}} = \frac{\abs{\p_1 \eta}^2 + \abs{\p_2 \eta}^2 }{1 + \sqrt{1+ \abs{\p_1 \eta}^2 + \abs{\p_2 \eta}^2}} \le \sqrt{\abs{\p_1 \eta}^2 + \abs{\p_2 \eta}^2}
\end{equation}
and hence
\begin{equation}\label{non_es_3}
 \norm{1 -  \abs{\n}}_{L^2(\Gamma)} \le \norm{\eta}_{H^1(\Gamma)}.
\end{equation}
The estimate \eqref{non_es_01} then follows by combining \eqref{non_es_2}--\eqref{non_es_3}.

\end{proof}

Next we present some identities identities for $J$ and $\a$.

\begin{lemma}\label{geometric_ids}
 Let $\a$, $J$, and $\n$ be as in \eqref{A_def}, \eqref{ABJ_def}, and \eqref{N_def}.  Then we have the following identities:
\begin{align}
\p_k (J\a_{ik})=0 \quad \text{ in }\Omega, \label{id1}\\
 \dt J  = \dt \bar{\eta}/d + \tilde{d} \dt \p_3 \bar{\eta}=\p_3(\dt \bar{\eta}   \tilde{d}  )\quad \text{ in }\Omega,\label{id2} \\
J \mathcal{A}_{j3} = \n_j\quad \text{ on } \Sigma, \label{id3} \\
J \mathcal{A}_{j3} = (e_3)_j\quad \text{ on } \Sigma. \label{id4}
\end{align}
\end{lemma}
\begin{proof}
 The first, third, fourth identities may be found in  Lemma 2.1 of \cite{GT_inf}.  The second is in the proof of Lemma 2.2 of the same paper. 
\end{proof}

Finally, we record an energy identity for smooth solutions of \eqref{geometric1} that motivates the definition of weak solution in Definition \ref{weak_moving_def}.

\begin{lemma}\label{mov_smooth_ident}
 Suppose $\theta$ is a smooth solution to \eqref{geometric1}.  Let $C_t^\beta$ and $F_t^\beta$ be as Definition \ref{weak_moving_def}.   Then
\begin{multline}\label{sid_00} 
\int_\Omega  J \dt \left(\theta -\theta_{eq}^\beta \right) v +  C_t^\beta(\theta-\theta_{eq}^\beta,v) + \int_\Omega \kappa J \naba (\theta-\theta_{eq}^\beta) \cdot \naba v \\
+ \int_\Omega  u \cdot \naba (\theta-\theta_{eq}^\beta) v J  - \dt \bar{\eta} \tilde{d} \p_3 (\theta-\theta_{eq}^\beta) v\\
  = F_t^\beta(v) + \int_\Omega \left( \dt \bar{\eta} \tilde{d} K \p_3 \theta_{eq}^\beta  -u_j \a_{jk} \p_k\theta_{eq}^\beta +\kappa \a_{jl} \p_l \left(\a_{jk} \p_k \theta_{eq}^\beta \right) \right) v J
\end{multline}
for every $v \in H^1_\beta(\Omega)$ (defined by \eqref{sobolev_beta}).

\end{lemma}
\begin{proof}
  We rewrite the $\theta$-equation in \eqref{geometric1} in a perturbed form around $\theta_{eq}^\beta$: 
\begin{equation}\label{sid_1}
\begin{split}
  \dt \left(\theta -\theta_{eq}^\beta \right)- \dt \bar{\eta} \tilde{d} K \p_3 
  \left( \theta  -\theta_{eq}^\beta \right)&+ u\cdot \naba \left( \theta  -\theta_{eq}^\beta \right) -\kappa 
\a_{jl} \p_l \left(\a_{jk} \p_k\left( \theta  -\theta_{eq}^\beta \right)\right) \\
&=    \dt \bar{\eta} \tilde{d} K \p_3 \theta_{eq}^\beta  -u_j \a_{jk} \p_k\theta_{eq}^\beta +\kappa \a_{jl} \p_l \left(\a_{jk} \p_k \theta_{eq}^\beta \right).
\end{split}
\end{equation}
We multiply \eqref{sid_1} by $J v$ and integrate over $\Omega$ to obtain the identity
\begin{equation}
 I + II + III = IV,
\end{equation}
where 
\begin{equation}
 I = \int_\Omega  J \dt \left(\theta -\theta_{eq}^\beta \right) v,
\end{equation}
\begin{equation}
 II =     \int_\Omega  u \cdot \naba (\theta-\theta_{eq}^\beta) v J - \dt \bar{\eta} \tilde{d} \p_3 (\theta-\theta_{eq}^\beta) v,
\end{equation}
\begin{equation}
 III = \int_\Omega - \kappa \a_{jl} \p_l \left(\a_{jk} \p_k\left( \theta  -\theta_{eq}^\beta \right)\right)  Jv,
\end{equation}
and
\begin{equation}
 IV = \int_\Omega \left( \dt \bar{\eta} \tilde{d} K \p_3 \theta_{eq}^\beta  -u_j \a_{jk} \p_k\theta_{eq}^\beta +\kappa \a_{jl} \p_l \left(\a_{jk} \p_k \theta_{eq}^\beta \right) \right) v J.
\end{equation}

To prove the desired equality it then suffices to rewrite $III$ by integrating by parts and employing Lemma \ref{geometric_ids}.  From \eqref{id1} we find that 
\begin{equation}\label{sid_2}
\begin{split}
III &=  \int_\Omega - \kappa \p_l \left(J\a_{jl}   \a_{jk} \p_k\left( \theta  -\theta_{eq}^\beta \right)\right)  v \\
& = \int_\Omega \kappa J\a_{jl}   \a_{jk} \p_k\left( \theta  -\theta_{eq}^\beta \right)  \p_l v  \\
& \quad + \int_{\Sigma_+} - \kappa J\a_{j3}   \a_{jk} \p_k\left( \theta  -\theta_{eq}^\beta \right)   v + \int_{\Sigma_-}  \kappa J\a_{j3}   \a_{jk} \p_k\left( \theta  -\theta_{eq}^\beta \right)   v.
\end{split}
\end{equation}
Then we use \eqref{id3} and \eqref{id4}  to rewrite 
\begin{equation}\label{sid_3}
 \int_{\Sigma_+} - \kappa J\a_{j3}   \a_{jk} \p_k\left( \theta  -\theta_{eq}^\beta \right)   v = \int_{\Sigma_+} - \kappa  \n \cdot \naba\left( \theta  -\theta_{eq}^\beta \right)   v
\end{equation}
and
\begin{equation}\label{sid_4}
 \int_{\Sigma_-}  \kappa J\a_{j3}   \a_{jk} \p_k\left( \theta  -\theta_{eq}^\beta \right)   v  = \int_{\Sigma_-}  \kappa   \a_{3k} \p_k\left( \theta  -\theta_{eq}^\beta \right)   v  = \int_{\Sigma_-}  \kappa   K \p_3 \left( \theta  -\theta_{eq}^\beta \right)   v 
\end{equation}
When $\beta \in [0,\infty)^2$ we may use the boundary conditions of \eqref{geometric_ids} to compute 
\begin{equation}\label{sid_5}
\begin{split}
 \kappa \frac{\n_j}{|\n|} \a_{jk}\p_k (\theta-\theta_{eq}^\beta) & = -\beta_+(\theta-\theta_{eq}^\beta)- \beta_+ |\n| K (\bar{\theta}-\theta_{eq}^\beta) + \beta_+(\bar{\theta}-\theta_{eq}^\beta) \\
 &=  -\beta_+(\theta-\theta_{eq}^\beta)- \beta_+\left( |\n| K -1\right)(\bar{\theta}-\theta_{eq}^\beta)  
 \end{split}
\end{equation}
on $\Sigma_+$ and 
\begin{equation}\label{sid_6}
 \kappa \p_3 \left( \theta  -\theta_{eq}^\beta \right)  = \beta_- \left( \theta  -\theta_{eq}^\beta \right)
\end{equation}
on $\Sigma_-$.  Plugging \eqref{sid_5} and \eqref{sid_6} into \eqref{sid_3} and \eqref{sid_4} and then replacing in \eqref{sid_2}  then yields \eqref{sid_00} when $\beta \in [0,\infty)^2$.  The remaining cases may be handled with similar computations.

\end{proof}

\end{document}